\documentclass[review]{elsarticle}

\usepackage{lineno,hyperref}
\hypersetup{hidelinks}
\usepackage{lipsum}
\usepackage{graphicx}
\usepackage{epstopdf}
\usepackage{algorithmic}
\usepackage{enumitem}
\usepackage{amsmath,amssymb,amsfonts,amsthm,epsfig,stmaryrd,latexsym,bm,diagbox,verbatim, mathrsfs}
\usepackage{color,xcolor}

\theoremstyle{plain}

\newtheorem{rmk}{Remark}[section]
\newtheorem{assp}{Assumption}[section]
\newtheorem{theorem}{Theorem}[section]
\newtheorem{coro}{Corollary}[section]
\newtheorem{lemma}{Lemma}[section]

\newtheorem{expl}{Example}[section]
\usepackage{cleveref}
\crefname{assp}{Assumption}{Assumptions}
\crefname{prop}{Proposition}{Propositions}
\crefname{theorem}{Theorem}{Theorems}
\crefname{lemma}{Lemma}{Lemmas}
\crefname{coro}{Corollary}{Corollaries}
\crefname{rem}{Remark}{Remarks}
\crefname{expl}{Example}{Examples}
\crefname{defn}{Definition}{Definitions}

\newcommand{\EE}{\mathbb{E}}

\newcommand{\NN}{\mathbb{N}}
\newcommand{\PP}{\mathbb{P}}
\newcommand{\RR}{\mathbb{R}}
\newcommand{\WW}{\mathbb{W}}

\newcommand{\BB}{\mathcal{B}}
\newcommand{\CC}{\mathcal{C}}

\newcommand{\CP}{\mathcal{P}}

\newcommand{\II}{\mathbf{1}}

                        \def\d{\mathrm{d}}        
     \def\f{\varphi}             
            
\def\m{\mu}         
      \def\p{\phi}

      \def\we{\wedge}

\def\D{\Delta}                  \def\K{\times}

\def\1{\oslash}   \def\2{\oplus}    \def\3{\otimes}      \def\4{\ominus}
\def\5{\circ}     \def\6{\odot}     \def\7{\backslash}   \def\8{\infty}
\def\9{\bigcap}   \def\0{\bigcup}   \def\+{\pm}          \def\-{\mp}

\usepackage{bm}
\usepackage[title,titletoc]{appendix}

\modulolinenumbers[5]
\allowdisplaybreaks[4]

\journal{Journal of \LaTeX\ Templates}

\begin{document}

\begin{frontmatter}

\title{Approximation of invariant measures for stochastic functional differential equations with infinite delay via an explicit finite-memory scheme\tnoteref{mytitlenote}}
\tnotetext[mytitlenote]{This work is funded by the National Natural Science Foundation of China (No. 12371402, No.12601747 ), the Tianjin Natural Science Foundation (No.24JCZDJC00830), the Royal Society (No. WM160014, Royal Society Wolfson Research Merit Award), the Royal Society of Edinburgh (No. RSE1832), and Education Department of Jilin Province, China (No. JJKH20261606KJ).}

\author[1]{Guozhen Li}
\author[2]{Guoting Song}
\author[1]{Shan Huang}
\author[3]{Xiaoyue Li\corref{cor1}}
\ead{lixy@tiangong.edu.cn}
\author[4]{Xuerong Mao}
\cortext[cor1]{Corresponding author.}

\affiliation[1]{organization={School of Mathematics and Statistics, Northeast Normal University},
    addressline={Changchun, Jilin 130024}, country={China.}}
\affiliation[2]{organization={School of Mathematics and Statistics, Changchun University},
    addressline={Changchun, Jilin 130024}, country={China.}}
\affiliation[3]{organization={School of Mathematical Sciences, Tiangong University},
    addressline={Tianjin 300387}, country={China.}}
\affiliation[4]{organization={Department of Mathematics and Statistics, University of Strathclyde},
    addressline={Glasgow G1 1XH}, country={U.K.}}

\begin{abstract}
Owing to the intrinsic computational limitation that infinite memory cannot be stored numerically, the approximation of invariant probability measures (IPMs) for stochastic functional differential equations (SFDEs) with infinite delay remains an open problem. This paper addresses this challenge by proposing an explicit finite-memory truncated Euler-Maruyama scheme that combines memory and state truncation for such SFDEs. We establish strong convergence of the numerical segment process over any finite time horizon and obtain a convergence order arbitrarily close to $1/2$. We further prove that this numerical segment process admits a unique IPM and that its IPM converges to the exact counterpart under the $L^2$-Wasserstein metric as the step size vanishes, with an explicit convergence rate. The framework applies to both linear and superlinear coefficients.

\end{abstract}


\end{frontmatter}

\linenumbers
\tableofcontents
\section{Introduction}

Stochastic functional differential equations (SFDEs) with infinite delay model randomly perturbed systems whose evolution depends on their entire past \cite{HMN91,WYM17}, with applications in mathematical biology, ecology, and finance \cite{AI05,WX09,YWK22}. 
Invariant probability measures (IPMs), which characterize statistical equilibria and determine stationary quantities, are fundamental to understanding their long-time behavior \cite{YWK22, BM05, NO26}.
Under suitable ergodicity conditions, they also link long-time averages along trajectories to expectations with respect to the invariant measure. 
Since invariant measures are rarely available in closed form, their numerical approximation is essential for the quantitative investigation of long-time dynamics. For SFDEs with infinite delay, however, the dependence on the entire past poses a fundamental computational challenge: a direct discretization of the full history requires infinitely many historical grid points and is therefore not implementable with finite storage. This motivates the development of computationally feasible finite-memory numerical schemes capable of accurately approximating the invariant measures of infinite-delay SFDEs.

The numerical approximation of IPMs for SFDEs mainly focuses on the finite-delay setting.
Under global Lipschitz conditions, Bao et al. \cite{BSY23} and Nguyen et al. \cite{NNN21} investigated the Euler-Maruyama (EM) scheme. For equations with superlinearly growing coefficients, Shi et al. \cite{SWMW24} studied a backward EM scheme, while Chen et al. \cite{CDHS24} considered a $\theta$-EM scheme. To ease computational burden, Li et al. \cite{LMS24} developed an explicit truncated EM scheme for approximating IPMs of a class of stochastic delay differential equations. In contrast, available numerical works targeting SFDEs with infinite-delay concentrate primarily on scheme design and strong convergence analysis. As one example, in our previous work \cite{LLMS23}, we introduced a computable finite-memory EM scheme and established its strong convergence over finite time horizons under global Lipschitz conditions. 
Zhou et al. \cite{ZXM26} and Li et al. \cite{LLM} subsequently developed explicit finite-memory truncated EM schemes for superlinear hybrid SFDEs with infinite delay and investigated their strong convergence.
These results provide a foundation for finite-memory numerical approximation of SFDEs with infinite delay.

Approximating the IPM of an SFDE with infinite delay amounts to a measure approximation problem on an infinite-dimensional fading-memory phase space. 
Finite-memory numerical segments depend only on a bounded history interval, whereas the original segment process depends on the entire past. 
This leads to a twofold approximation: a finite-history function space approximates the original infinite-history function space, while the numerical IPM defined on the finite-history space approximates the IPM defined on the original infinite-history space. 
To recover the infinite-history limit, the memory-truncation length must tend to infinity. 
However, existing analytical approaches for SFDEs with finite delay cannot be directly extended to this setting, since the resulting estimates may depend on the delay length. 
This necessitates moment bounds and contractivity estimates that are uniform with respect to the memory-truncation length. 
Moreover, superlinearly growing coefficients may lead to moment divergence in the standard Euler–Maruyama approximation \cite{HJK11}, making the analysis of numerical ergodicity more challenging. 
Due to these challenges, the numerical approximation of IPMs for SFDEs with infinite delay remains an open problem.

In this paper,  we consider the following $n$-dimensional SFDEs of the form
\begin{equation}\label{ISFDE}
\mathrm dx(t)=f(x_t)\,\mathrm dt+g(x_t)\,\mathrm dB(t),
\qquad t>0,
\tag{1.1}
\end{equation}
with initial data \(x_0=\xi\in\mathcal C_r\), where
$
x_t:=\{x(t+u):u\in\mathbb R_-\}
$ with $t\ge0$
denotes the segment and \(\mathbb R_-:=(-\infty,0]\), the mappings \(f:\mathcal C_r\to\mathbb R^n\) and \(g:\mathcal C_r\to\mathbb R^{n\times d}\) describe the dependence of the coefficients on the history, and \(B(t)\) is a \(d\)-dimensional Brownian motion.
Here, the fading-memory phase space \(\mathcal C_r\) is defined by 
\[
\mathcal C_r
=
\left\{
\varphi\in C(\mathbb R_-;\mathbb R^n):
\lim_{u\to-\infty}e^{ru}\varphi(u)
\text{ exists in }\mathbb R^n
\right\},
\tag{1.2}
\]
equipped with the norm
$
\|\varphi\|_r
=
\sup_{-\infty<u\le0}e^{ru}|\varphi(u)|
$
with \(r>0\), and is a Polish space; see \cite{HMN91,MN89} for details.
\(C(\mathbb R_-;\mathbb R^n)\) denotes the space of continuous functions from \(\mathbb R_-\) to \(\mathbb R^n\).
For \eqref{ISFDE}, we construct an explicit finite-memory truncated EM scheme that combines memory and state truncation, establish its strong convergence, and derive a convergence rate. 
Moreover, we approximate the IPM using the associated numerical segment process.

To address the difficulties arising from infinite delay, we fully exploit the exponential fading memory structure of the phase space to derive estimates that are uniform with respect to the memory-truncation length. In addition, we introduce state truncation to control the excessive growth of the numerical segment process caused by superlinearly growing coefficients and establish the long-time estimates required for numerical ergodicity.
Our main contributions concern both finite-time convergence and the long-time behavior of the numerical scheme, as summarized below.
\begin{itemize}
\item Under weaker Khasminskii-type conditions, we establish the strong convergence of the explicit finite-memory numerical scheme on any finite time interval and derive the convergence rate of order $1/2-\varepsilon$ for any $\varepsilon \in (0, 1/2)$.

\item Under suitable dissipativity conditions, we prove that the finite-memory numerical functional solution process admits a unique IPM and is exponentially ergodic in the \(L^2\)-Wasserstein distance.

\item As the memory truncation parameter tends to infinity and the time step size tends to zero, we prove that the numerical IPM converges to the exact IPM  and derive a explicit convergence rate. To our knowledge, this is the first work on invariant measure approximation for SFDEs with infinite delay, and our theory covers both linear and nonlinear cases.
\end{itemize}

The remainder of this paper is organized as follows. Section 2 introduces the notation and preliminary results. Section 3 constructs the finite-memory TEM scheme and establishes the required moment boundedness. Section 4 proves strong convergence of the numerical segment process on finite time intervals and derives its convergence rate. Section 5 studies the ergodicity of the exact and numerical segment processes and establishes convergence of the numerical IPM to the exact IPM, together with an explicit convergence rate. Section 6 presents numerical experiments illustrating the theoretical results. Appendix \ref{appA} gives the proofs of Theorem \ref{th2.3}, Lemmas \ref{l3.6} and \ref{lemma3.6}.

\section{Preliminaries}\label{S2}
We begin by introducing the notation used throughout the paper. Let $d$, $m$, and $n$ be positive integers.
Let $\mathbb{R}_+ = [0, \infty)$ and $\mathbb{R}_- = (-\infty, 0]$, and let $\mathbb Z$, $\mathbb{N}$, and $\mathbb{N}_+$ denote the sets of integers, nonnegative integers, and positive integers, respectively. Let  $\mathbb{R}^n$ be the $n$-dimensional Euclidean space equipped with the standard Euclidean norm $|\cdot|$, and let $\mathcal{B}(\mathbb{R}^n)$ denote its Borel $\sigma$-algebra.
If $a, b \in \mathbb{R}$, define $a \wedge b = \min\{a, b\}$ and $a\vee b=\max\{a,b\}$.
For $a\in\mathbb R$, 
denote by $\lfloor a \rfloor$ and $\lceil a\rceil$ the largest integer not greater than $a$ and the smallest integer not less than $a$, respectively. If $A$ is a vector or matrix, its transpose is denoted by $A^{\mathrm{T}}$.
For a set $\mathbb{D}$, its indicator function is denoted by $\mathbf{1}_\mathbb{D}$; that is, $\mathbf{1}_\mathbb{D}(x) = 1$ if $x \in \mathbb{D}$ and $0$ otherwise. Throughout the paper, $C$ denotes a generic positive constant that may vary from line to line, whereas $C_\zeta$ may depend on $\zeta$ but is independent of all other relevant variables.

Let $(\Omega, \mathcal{F}, \mathbb{P})$ be a complete probability space equipped with a filtration $\{\mathcal{F}_t\}_{t \ge 0}$ satisfying the usual conditions.  Let $\mathbb{E}$ denote expectation with respect to $\mathbb{P}$, and let $\{B(t)\}_{t \ge 0}$ be a $d$-dimensional Brownian motion on this space.
Denote by $\mathcal{P}_0$ the family of probability measures $\mu$ on $\mathbb{R}_-$. For each $a>0$, define
$$
\mathcal{P}_a=\bigg\{ \mu\in\mathcal{P}_0: \mu^{(a)}:= \int_{-\infty}^0 e^{-au}\mu(\d u)<\infty \bigg\}.
$$
Clearly, $\mathcal{P}_{\tilde a} \subset \mathcal{P}_a\subset \mathcal{P}_0$ whenever $\tilde a >a >0$. Moreover, if $\mu\in \mathcal{P}_{\tilde a}$, then $\mu^{(a)}$ is nondecreasing and continuous in $a\in[0,\tilde a]$, and $\lim_{a \rightarrow 0} \mu^{(a)} = \mu^{(0)} = 1$; see \cite[Lemma 2.1]{WYM17}.

We impose the following conditions to ensure the existence and uniqueness of global solutions.

\begin{assp}\label{a2.1-f}
For each $R > 0$, there exists a positive constant $K_R$ such that for any $\phi, \varphi \in \mathcal{C}_r$ with $\|\phi\|_r\vee\|\varphi\|_r \le R$,
\begin{equation*}
|f(\phi)-f(\varphi)|  \le
  K_R   \|\phi-\varphi\|_r.
\end{equation*}
\end{assp}

\begin{assp}\label{a2.1-g}
For each $R > 0$, 
there exists a positive constant $K'_R$ such that for any $\phi, \varphi \in \mathcal{C}_r$ with $\|\phi\|_r\vee\|\varphi\|_r \le R$,
\begin{equation*}
|g(\phi)-g(\varphi)|  \le
  K'_R   \|\phi-\varphi\|_r.
\end{equation*}
\end{assp}

\begin{assp}\label{a2.2}
There exist constants $a_1 > 0$, $a_2 \ge 0$, $\bar p > 2$ and a probability measure $\mu_1 \in \CP_{\bar p r}$ such that for any $\phi \in \mathcal{C}_r$,
$$
\langle \phi(0), f(\phi) \rangle  \le a_1 \left( 1 + \|\p\|_r^2 \right) -a_2 |\phi(0)|^{ \bar p} + a_2 \int_{-\infty}^0 |\phi(u)|^{\bar p} \mu_1(\d u)
$$
and
$$
   |g(\phi)|^2 \le a_1 \left( 1 + \|\p\|_r^2 \right).
$$
\end{assp}

The first inequality in Assumption \ref{a2.2} is a generalized Khasminskii-type condition that allows for a broader class of SFDEs with superlinear coefficients and infinite delay. In the remainder of the paper, whenever Assumption \ref{a2.2} is invoked, we restrict our attention to the nontrivial case $a_2>0$, as the case $a_2=0$ follows from the same arguments with straightforward modifications.

\begin{theorem}\label{th2.3}
Let Assumptions \ref{a2.1-f}\text{-}\ref{a2.2} hold. Then the SFDE with infinite delay \eqref{ISFDE} with initial data $\xi \in \mathcal{C}_r$ admits a unique global solution $x(t)$ on $t \in (-\infty, \infty)$, and satisfies that for any $T>0$ and $p > 0$,
$$
\mathbb{E}\Big[\sup_{0\le t\le T}\|x_t\|_r^p\Big]   \le C_{T, \xi},
$$
where $C_{T,\xi} = C_{\xi} e^{CT^{\frac{p}{2} \vee 1}}$ when $a_2 > 0$ and $C_{T,\xi} = C_{\xi} e^{CT}$ when $a_2 = 0$.
Moreover, for any constant $h > \|\xi\|_r$, 
\begin{equation*}
\PP\{\tau_h \le T \}  \le \frac{C_{T,\xi}}{h^{p}},
\end{equation*}
where
$\tau_h := \inf\{t \ge 0: |x(t)|\ge h \}.$
\end{theorem}

Since the proof follows from standard localization arguments and moment estimates, we provide only a brief sketch in Appendix \ref{appA}.

\section{Explicit numerical scheme and moment boundedness}\label{S3}
This section develops an explicit finite-memory truncated Euler--Maruyama (TEM) scheme for the SFDE with infinite delay \eqref{ISFDE} and  establishes the moment boundedness of numerical functional solution. 
We first present the detailed construction of the explicit finite-memory TEM scheme in three steps.

\textbf{Step 1. Memory-truncation setup.}
We truncate the infinite delay. Let $K\in\mathbb{N}_+$ denote the memory-truncation length, and $\Delta\in(0,1]$ with
$\Delta^{-1}\in\mathbb{N}_+$ denote the time step size. It is straightforward to see that
$k:= K/\Delta \in\mathbb{N}_+$.
Throughout the paper, every statement involving $K$ and $\Delta$ is understood to use step sizes satisfying these restrictions.

\textbf{Step 2. State-truncation setup.}
By Assumption \ref{a2.1-f}, we first choose a continuous, strictly increasing function
$\Gamma:[0,\infty)\to\mathbb{R}_+$ such that $\lim_{R \to +\infty} \Gamma(R) = +\infty$ and for each $R > 0$, 
\begin{equation}\label{trun}
|f(\phi)-f(\varphi)|
\leq \Gamma(R)\|\phi-\varphi\|_r
\end{equation}
for all $\phi,\varphi\in\mathcal{C}_r$ satisfying
$\|\phi\|_r\vee\|\varphi\|_r\leq R$. 
Denote by $\Gamma^{-1}:[\Gamma(0),+\infty)\rightarrow\mathbb{R}_+$ the inverse function of $\Gamma$. 
Define the truncation mapping $\Pi^{\Delta}:\mathbb{R}^n\rightarrow \mathbb{R}^n$ by
\begin{align*}
\Pi^{\Delta}(x)=\bigg( |x|\wedge\Gamma^{-1}\left(L \Delta^{-\theta}\right)  \bigg)\frac{x}{|x|}, 
\end{align*}
where $\theta \in (0,1/2)$ is a constant and 
$L > 1 \vee |f(\mathbf{0})| \vee \Gamma(0)$
. Here, we use the convention $x/|x|=\mathbf{0}$ when $x=\mathbf{0}$.

\textbf{Step 3. Construction of the finite-memory TEM scheme.}
We now define the finite-memory numerical scheme for the superlinear SFDE with infinite delay \eqref{ISFDE}. Let $t_j=j\D$ for every integer $j$, and define
\begin{align}\label{TEM}
\begin{cases}
Y^{K,\Delta}(t_j)=\xi(t_j), ~ j = -k, \ldots, 0,\\
X^{K,\Delta}({t_j})=\Pi^{\Delta}(Y^{K,\Delta}(t_j)), ~j = -k, \ldots, 0, 1, \ldots, \\
Y^{K,\Delta}(t_{j+1})=X^{K,\Delta}(t_j)+f(X^{K,\Delta}_{t_j})\D
+ g( X^{K,\Delta}_{t_j}) \Delta B_j, ~ j=0,1, \ldots, \\
\end{cases}
\end{align}
where $\D B_j=B(t_{j+1})-B(t_j)$ and $X^{K,\Delta}_{t_j}$ is a $\CC_r$-valued random variable defined by
\begin{equation}\label{LI}
X^{K,\D}_{t_j}(u)=\left\{
\begin{aligned}
&\frac{t_{m+1}-u}{\D} X^{K,\D}(t_{j+m}) + \frac{u-t_m}{\D} X^{K,\D}(t_{j+m+1}), \\
&~~~~~~~~~~~~~~~~~u\in[t_m, t_{m+1}], m\in\{-k,-k+1,\ldots, -1\}, \\
&X^{K,\D}(t_{j-k}), \quad u < -K.
\end{aligned}\right.
\end{equation}

We refer to this method as the \emph{finite-memory truncated Euler--Maruyama (TEM) scheme}.
In addition, define the continuous-time numerical solution and segment processes by
\begin{equation}\label{step}
\begin{gathered}
Y^{K,\D}(t) = \sum_{j \geq -k}Y^{K,\D}(t_j) \II_{ [t_j, t_{j+1})}(t) + Y^{K,\D}(-K) \II_{ (-\infty, -K)}(t), \\
X^{K,\D}(t) = \sum_{j \geq -k}X^{K,\D}(t_j) \II_{ [t_j, t_{j+1})}(t) + X^{K,\D}(-K) \II_{ (-\infty, -K)}(t), \\
X^{K,\D}_t = \sum_{j \geq 0} X^{K,\D}_{t_j} \II_{ [t_j, t_{j+1})}(t).
\end{gathered}
\end{equation}
We call $\{X^{K,\D}(t)\}_{t \in\mathbb R}$ and $\{X^{K,\D}_t\}_{t \in\mathbb R_+}$ the {\emph{finite-memory TEM numerical solution process}} and the {\emph{numerical segment (or functional solution) process}}, respectively.

{\color{black}\begin{rmk}\label{rmk3.1}
The phase space $\mathcal{C}_r$ is equipped with an exponentially weighted supremum norm and therefore has a fading-memory structure.
Using this property, we construct a finite-memory numerical scheme by truncating the infinite history to the interval $[-K,0]$. At each iteration, the numerical segment $X_{t_j}^{K,\Delta}$ is constructed from only the most recent $k+1$ grid values.
\end{rmk}}

\begin{rmk}
It is known that the standard EM method may exhibit moment divergence for stochastic differential equations with superlinearly growing coefficients (see, e.g., \cite{HJK11}). The finite-memory TEM scheme \eqref{TEM} modifies the EM method by truncating the state, thereby controlling the discretization of the superlinear drift through
\begin{equation}\label{f-lin}
| f \big(X^{K,\Delta}_t\big)  |     \le   L \Delta^{-\theta} ( 1+   \|X^{K,\Delta}_t \|_r ), \quad \forall\, t \ge 0.
\end{equation}
If $f$ and $g$ are globally Lipschitz continuous, the truncation can be omitted by taking $\Pi^\Delta(x)=x$. In this case, the proposed TEM scheme reduces to the EM scheme for SFDEs with infinite delay studied in \cite{LLMS23}.
\end{rmk}

\begin{lemma}\label{th3.1}
	Let Assumptions \ref{a2.1-f}\text{-}\ref{a2.2} hold. Then, for any $T > 0$ and $p>0$,
$$
\sup_{K \in \NN_+} \sup_{0 < \D \le 1} \EE \Big[ \sup_{0 \le t \le T} \|X^{K,\Delta}_t \|_r^p \Big] \le C_{T,\xi},
$$
where $C_{T,\xi}$ is defined in Theorem \ref{th2.3}.
\end{lemma}

\begin{proof}
By H\"older's inequality, it suffices to prove the result for $p\geq2$. Fix $K \in \NN_+$, $T > 0$ and $\Delta \in (0, 1]$ arbitrarily. 
It follows from  \eqref{TEM} that
	$$
	\begin{aligned}
	& \quad\,\, |X^{K,\Delta}(t_{j+1})|^2 \\
    & \le |X^{K,\Delta}(t_j)|^2 + 2 \langle X^{K,\D}(t_j),  f(X^{K,\D}_{t_j}) \rangle \D  +  | f(X^{K,\D}_{t_j}) |^2 \D^2 + | g(X^{K,\D}_{t_j}) \Delta B_j |^2  \\
		& \quad\,\, + 2 \big\langle X^{K,\D}(t_j) ,  g(X^{K,\D}_{t_j}) \D B_j \big\rangle  +  2 \big\langle f(X^{K,\D}_{t_j}) \D,  g(X^{K,\D}_{t_j}) \D B_j \big\rangle.
	\end{aligned}
	$$
	Assume without loss of generality that $T\geq\D$, and let
	$1\leq n\leq\lfloor T/\D\rfloor$. Summing the preceding inequality over
	$j=0,\ldots,n-1$ and using Assumption~\ref{a2.2}, \eqref{f-lin}, and
	$\theta\in(0,1/2)$, we obtain
\begin{align}\nonumber
	& \quad |X^{K,\Delta}(t_n)|^2 
	 \le \|\xi\|_r^2 + CT
	  + C\Delta \sum_{j=0}^{n-1}\|X^{K,\Delta}_{t_j}\|_r^2
	  - 2a_2\Delta \sum_{j=0}^{n-1}|X^{K,\Delta}(t_j)|^{\bar p} \\ \nonumber
	& \quad + 2a_2\Delta \sum_{j=0}^{n-1}\int_{-\infty}^0
	    |X^{K,\Delta}_{t_j}(u)|^{\bar p}\mu_1(\d u)
	  + C\sum_{j=0}^{n-1}\left(1+\|X^{K,\Delta}_{t_j}\|_r^2\right)|\Delta B_j|^2 \\ \nonumber
	& \quad + C\Delta^{1-\theta}\sum_{j=0}^{n-1}
	    \left(1+\|X^{K,\Delta}_{t_j}\|_r^2\right)|\Delta B_j|
	  + 2\left|\sum_{j=0}^{n-1}\big\langle X^{K,\D}(t_j),
	    g(X^{K,\D}_{t_j})\Delta B_j\big\rangle\right|. \nonumber
\end{align}
	The argument used in the proof of \cite[Theorem 3.1]{LLM} gives
	$$
	\Delta \sum_{j=0}^{n-1} \int_{-\infty}^0 |X^{K,\Delta}_{t_j}(u)|^{\bar p} \mu_1(\d u) \le C \|\xi\|_r^{\bar p} + \Delta \sum_{j=0}^{n-1} |X^{K,\Delta}(t_j)|^{\bar p}.
	$$
	Therefore, 
	$$
	\begin{aligned}
		& \quad\,\, |X^{K,\Delta}(t_n)|^2 \\
		& \le C \left( \|\xi\|_r^2 + \|\xi\|_r^{\bar p} \right) + C T
		  + C \Delta \sum_{j=0}^{n-1} \|X^{K,\Delta}_{t_j}\|_r^2  + C \sum_{j=0}^{n-1} \left( 1 + \|X^{K,\Delta}_{t_j}\|_r^2 \right) |\Delta B_j|^2
		 \\
        & \quad + C \Delta^{1-\theta} \sum_{j=0}^{n-1} \left( 1 + \|X^{K,\Delta}_{t_j}\|_r^2 \right) |\Delta B_j| + 2 \left| \sum_{j=0}^{n-1} \big\langle X^{K,\D}(t_j),
		  g(X^{K,\D}_{t_j}) \Delta B_j \big\rangle \right|. \\
	\end{aligned}
	$$
	Taking the $p/2$-th power of both sides gives
	$$
	\begin{aligned} |X^{K,\Delta}(t_n)|^p    
		& \le C \left( \|\xi\|_r^p + \|\xi\|_r^{\frac{p\bar p}{2}} \right) + C T^{\frac{p}{2}}
		  + C T^{\frac{p}{2}-1} \Delta \sum_{j=0}^{n-1} \|X^{K,\Delta}_{t_j}\|_r^p \\
		& \quad\,\, + C n^{\frac{p}{2}-1} \sum_{j=0}^{n-1}
		  \left( 1 + \|X^{K,\Delta}_{t_j}\|_r^p \right) |\Delta B_j|^p \\
		& \quad\,\, + C T^{\frac{p}{2}-1} \Delta^{1-\frac{p\theta}{2}} \sum_{j=0}^{n-1}
		  \left( 1 + \|X^{K,\Delta}_{t_j}\|_r^p \right) |\Delta B_j|^{\frac{p}{2}} \\
		& \quad\,\, + 6^{\frac{p}{2}-1}2^{\frac{p}{2}}
		  \left| \sum_{j=0}^{n-1} \big\langle X^{K,\D}(t_j),
		  g(X^{K,\D}_{t_j})\Delta B_j \big\rangle \right|^{\frac{p}{2}}. \\
	\end{aligned}
	$$
	Hence, for any $1 \le N \le \lfloor T/\D \rfloor$, 
\begin{equation}\label{3.15}
\begin{aligned}
& \quad\,\, \EE \Big[ \sup_{0 \le n \le N}\left|X^{K,\D}(t_n)\right|^{p} \Big] \\
	& \le C \left( \|\xi\|_r^{p} + \|\xi\|_r^{\frac{p \bar p}{2}} \right)  +  CT^{\frac{p}{2}}  + C T^{\frac{p}{2}-1} \D \sum_{j=0}^{N-1} \EE \big[\|X^{K,\D}_{t_j}\|_r^{p} \big]  \\
	& \quad\,\, +  6^{\frac{p}{2}-1} 2^{\frac{p}{2}} \EE \left[ \sup_{0 \le n \le N}\left|   \sum_{j=0}^{n-1} \big\langle X^{K,\D}(t_j) ,  g(X^{K,\D}_{t_j}) \D B_j \big\rangle \right|^{\frac{p}{2}} \right]. \\
	\end{aligned}
	\end{equation}
By the standard argument (see for example \cite[Chapter 5, (4.6)]{M08}) and Assumption \ref{a2.2}, we obtain
\begin{equation}
\begin{aligned}\label{3-3-16}
	& \quad\,\, 6^{\frac{p}{2}-1}2^{\frac{p}{2}} \EE \left[ \sup_{0 \le n \le N}\left|   \sum_{j=0}^{n-1} \big\langle X^{K,\D}(t_j) ,  g(X^{K,\D}_{t_j}) \D B_j \big\rangle \right|^{\frac{p}{2}} \right] \\
	& \le \frac{e^{-pr}}{2} \EE  \Big[ \sup_{0\le n \le N} \| X^{K,\Delta}_{t_n}\|_r^{p} \Big]  + C T^{\frac{p}{2}} +  C T^{\frac{p}{2}-1} \Delta \sum_{j=0}^{N-1} \EE \big[ \|X^{K,\Delta}_{t_j}\|_r^{p} \big].
\end{aligned}
\end{equation}
Substituting \eqref{3-3-16} into \eqref{3.15} gives
	\begin{equation}\label{3-3.16}
	\begin{aligned}
\EE \Big[ \sup_{0 \le n \le  N}\left|X^{K,\D}(t_n)\right|^{p} \Big] 
	& \le \frac{e^{-pr}}{2} \EE\Big[\sup_{0 \le n \le N} \|X^{K,\Delta}_{t_n}\|_r^p \Big] +   C \left( \|\xi\|_r^{p} + \|\xi\|_r^{\frac{p \bar p}{2}} \right) \\
    & \quad +  C  T^{\frac{p}{2}} +  C T^{\frac{p}{2}-1} \D \sum_{j=0}^{N-1} \EE \big[ \|X^{K,\D}_{t_j}\|_r^{p} \big].
	\end{aligned}
	\end{equation}
By \eqref{LI} and the definition of the norm $\|\cdot\|_r$, we have
	$
	\| X^{K,\Delta}_{t_n}\|_r^{p} \le e^{p r \Delta} \Big(\|\xi\|_r^{p} + \sup_{0 \le j \le n}|X^{K,\Delta}(t_j)|^{p}  \Big)
	$
	for any $n \ge 0$. Combining this with \eqref{3-3.16}, we conclude that
	$$
	\mathbb{E} \Big[ \sup_{0 \le n \le N} \| X^{K,\Delta}_{t_n}\|_r^p \Big] \le  C \left( \|\xi\|_r^{p} + \|\xi\|_r^{\frac{p \bar p}{2}} \right)  +  C  T^{\frac{p}{2}} +  C T^{\frac{p}{2}-1} \D \sum_{j=0}^{N-1} \EE \big[ \|X^{K,\D}_{t_j}\|_r^{p} \big].
	$$
	The desired assertion follows from the discrete Gronwall inequality (see \cite[Theorem 2.5, p.56]{MY06}).
\end{proof}

\section{Strong convergence and convergence rate}\label{sec3.2}
This section establishes the strong convergence, over any finite time horizon, of the finite-memory TEM approximation defined in \eqref{TEM} to the exact solution. Under an additional polynomial growth condition, we further derive an explicit convergence rate in terms of the memory-truncation parameter and the time step size.

\subsection{Strong convergence}
We first introduce a finite-memory auxiliary equation proposed in \cite{LLMS23} for approximating the SFDE with infinite delay \eqref{ISFDE}.
For any $K \in \mathbb N_+$, define a mapping $\pi_K:\mathcal{C}_r\to\mathcal{C}_r$ by
$$
\pi_K(\f)(u)=
\begin{cases}
  \f(u), & \mbox{if } u \in [-K,0], \\
  \f(-K),  & \mbox{if }   u \in (-\infty,-K).
\end{cases}
$$
The auxiliary equation is defined by
\begin{equation}\label{TSFDE}
 \d x^K(t) = f_K(x^K_t) \d t + g_K(x^K_t)\d B(t), \quad t > 0 
\end{equation}
with initial data $x^K_0 = \xi$, where $f_K:\mathcal{C}_r \rightarrow \mathbb{R}^n$ and $g_K:\mathcal{C}_r \rightarrow \mathbb{R}^{n \K d}$ are given by
$
f_K(\f):=f(\pi_K(\f)),~ g_K(\f):=g(\pi_K(\f)).
$
Clearly, equation \eqref{TSFDE} can be written in integral form as
$$
\left\{
\begin{aligned}
 & x^K(t) = \xi(0) +  \int_0^t f_K(x^K_s) \d s + \int_0^t g_K(x^K_s)\d B(s), \quad t > 0, \\
 & x^K_0 = \xi.
\end{aligned}\right.
$$
Although the initial data $\xi$ belongs to $\CC_r$, the coefficients $f_K(\varphi)$ and $g_K(\varphi)$ depend only on the restriction of $\varphi$ to  $[-K,0]$. Thus, \eqref{TSFDE} is an SFDE with finite delay and provides a finite-memory approximation to \eqref{ISFDE}. 

We further introduce a finite-memory numerical auxiliary process $\{Z^{K,\Delta}(t)\}_{t \geq 0}$ as follows:
\begin{equation}\label{ap}
\left\{
\begin{aligned}
&Z^{K,\D}(t) = \xi(t), \quad t < 0,\\
&Z^{K,\D}(t) = X^{K,\D}(t_j) + f(X^{K,\D}_{t_j})(t-t_j) \\
& \qquad + g(X^{K,\D}_{t_j})(B(t)-B(t_j)),\ t \in [t_j, t_{j+1}),\ j \in \NN 
\end{aligned}\right.
\end{equation}
and
$$
Z^{K,\D}_t(u)=\left\{
\begin{aligned}
& Z^{K,\D}(t+u),\quad u \in [-K,0],\\
& Z^{K,\D}(t-K), \quad u \in (-\infty, -K).
\end{aligned}\right.
$$
Observe that for any $t \in [t_j, t_{j+1})$ with $j \in \NN$,
$$
 Z^{K,\D}(t) = X^{K,\D}(t_j) + f_K(X^{K,\D}_{t_j})(t-t_j) + g_K(X^{K,\D}_{t_j})(B(t)-B(t_j))
$$
and
\begin{equation}\label{5-12}
Z^{K,\D}(t_j) = X^{K,\D}(t_j), \quad \lim_{t\to t_j^-}Z^{K,\D}(t) = Y^{K,\D}(t_j).
\end{equation}
Thus, $Z^{K,\Delta}(t)$ describes the one-step evolution of the numerical scheme before the state-truncation mapping $\Pi^\Delta$ is applied, which allows us to estimate the difference between $Z_t^{K,\Delta}$ and $X_t^{K,\Delta}$.
Moreover, for any $\Delta_1 \in (0,1]$ and $\Delta \in (0, \Delta_1]$, define the stopping time
\begin{equation}\label{rh}
\tau^K_{\Delta,\Delta_1} = \inf \left\{t\ge0 : \big|Z^{K,\D}(t)\big| \geq \Gamma^{-1}(L \Delta_1^{-\theta}) \right\}.
\end{equation}
By \eqref{5-12}, \eqref{rh}, and $\Delta\leq\Delta_1$, we have
\begin{equation}\label{+Z}
\begin{aligned}
 Z^{K,\Delta}(t \wedge \tau_{\Delta,\Delta_1}^K) & =
X^{K,\Delta}(0)
+\int_0^{t \wedge \tau_{\Delta,\Delta_1}^K} f_K(X_s^{K,\Delta})\,\d s \\
& \quad \,\, +\int_0^{t \wedge \tau_{\Delta,\Delta_1}^K} g_K(X_s^{K,\Delta})\,\d B(s)
\end{aligned}
\end{equation}
for any $t \geq 0$, where $X_s^{K,\Delta}$ is defined in \eqref{step}. This integral representation helps us to estimate the difference between $Z_t^{K,\Delta}$ and $x_t^K$.

The following lemmas give moment boundedness for the segment process $\{x_t^K\}_{t \geq 0}$ and auxiliary process $\{Z^{K,\Delta}(t)\}_{t \geq 0}$.

\begin{lemma}\label{th2.3'}
Let Assumptions \ref{a2.1-f}\text{-}\ref{a2.2} hold. Then the SFDE \eqref{TSFDE} with initial data $x^K_0 = \xi$ admits a unique global solution $x^K(t)$ on $t \in (-\infty, \infty)$. Moreover, for any $T > 0$ and $p > 0$,
$$
 \sup_{K \in\mathbb N_+ } \mathbb{E} \Big[ \sup_{0\le t \le T} \|x^K_t\|_r^p \Big] \le C_{T, \xi},
$$
where $C_{T,\xi}$ is defined in Theorem \ref{th2.3}. Furthermore,
for any $h > \|\xi\|_r$,
\begin{equation*}
 \sup_{K\in\mathbb N_+}\PP\{\tau^K_h \le T \}  \le \frac{C_{T,\xi}}{h^{p}}, 
\end{equation*}
where 
$\tau^K_h := \inf\{t \ge 0: |x^K(t)|\ge h\}.$
\end{lemma}
\begin{proof}
The contraction property $\|\pi_K(\phi) - \pi_K(\varphi)\|_r\leq \|\phi-\varphi\|_r$ implies that $f_K$ and $g_K$ inherit the local Lipschitz and growth conditions of $f$ and $g$, with the corresponding constants independent of $K$. Moreover, $f_K$ and $g_K$ satisfy the corresponding generalized Khasminskii-type condition. Therefore, the standard localization arguments and moment estimates used in the proof of \cite[Theorem~3.1]{LLMS23} yield the desired assertions.
\end{proof}

\begin{lemma}\label{l3.5}
	Let Assumptions \ref{a2.1-f}\text{-}\ref{a2.2} hold. Then for any $T > 0$ and $p > 0$,
	$$
	\sup_{K \in\mathbb N_+}\sup_{0 < \D \le 1}\EE \Big[\sup_{0 \le t\le T} |Z^{K,\D}(t)|^p \Big]  \le C_{T,\xi}.
	$$
	Moreover, for any $\Delta_1 \in (0, 1]$ and $\Delta \in (0, \Delta_1]$,
	$$
	\sup_{K \in\mathbb N_+} \sup_{0 < \D \le \Delta_1} \mathbb{P}\{\tau^K_{\Delta,\Delta_1}\le T\} \le \frac{C_{T,\xi}}{\big(\Gamma^{-1}(L \Delta_1^{-\theta})\big)^p },
	$$
	where $\tau^K_{\Delta,\Delta_1}$ and $C_{T,\xi}$ are defined in \eqref{rh} and Theorem \ref{th2.3}, respectively.
\end{lemma} 
\begin{proof}
The proof of this Lemma follows the arguments used for Lemma~\ref{th3.1} and \cite[Lemma~3.3]{LLM}, so we omit the details.
\end{proof}

With the auxiliary processes $\{x_t^K\}_{t\geq 0}$ and
$\{Z_t^{K,\Delta}\}_{t\geq 0}$ introduced above, the analysis of the strong convergence of the finite-memory TEM scheme,
namely, the estimation of
$
\mathbb{E}\big[\|X_t^{K,\Delta}-x_t\|_r^q\big],
$
amounts to estimating
$
\mathbb{E}\big[\|X_t^{K,\Delta}-x_t^K\|_r^q\big]$
and
$\mathbb{E}\big[\|x_t^K-x_t\|_r^q\big].
$
Here, the first term is analyzed with the aid of the auxiliary process
$\{Z_t^{K,\Delta}\}_{t\geq 0}$, together with a localization argument
based on stopping times. The overall convergence argument is summarized
in Figure~\ref{3-tu}. In what follows, we estimate these two terms
separately.
\begin{figure}[!htp]
\centering
\includegraphics[width=10cm]{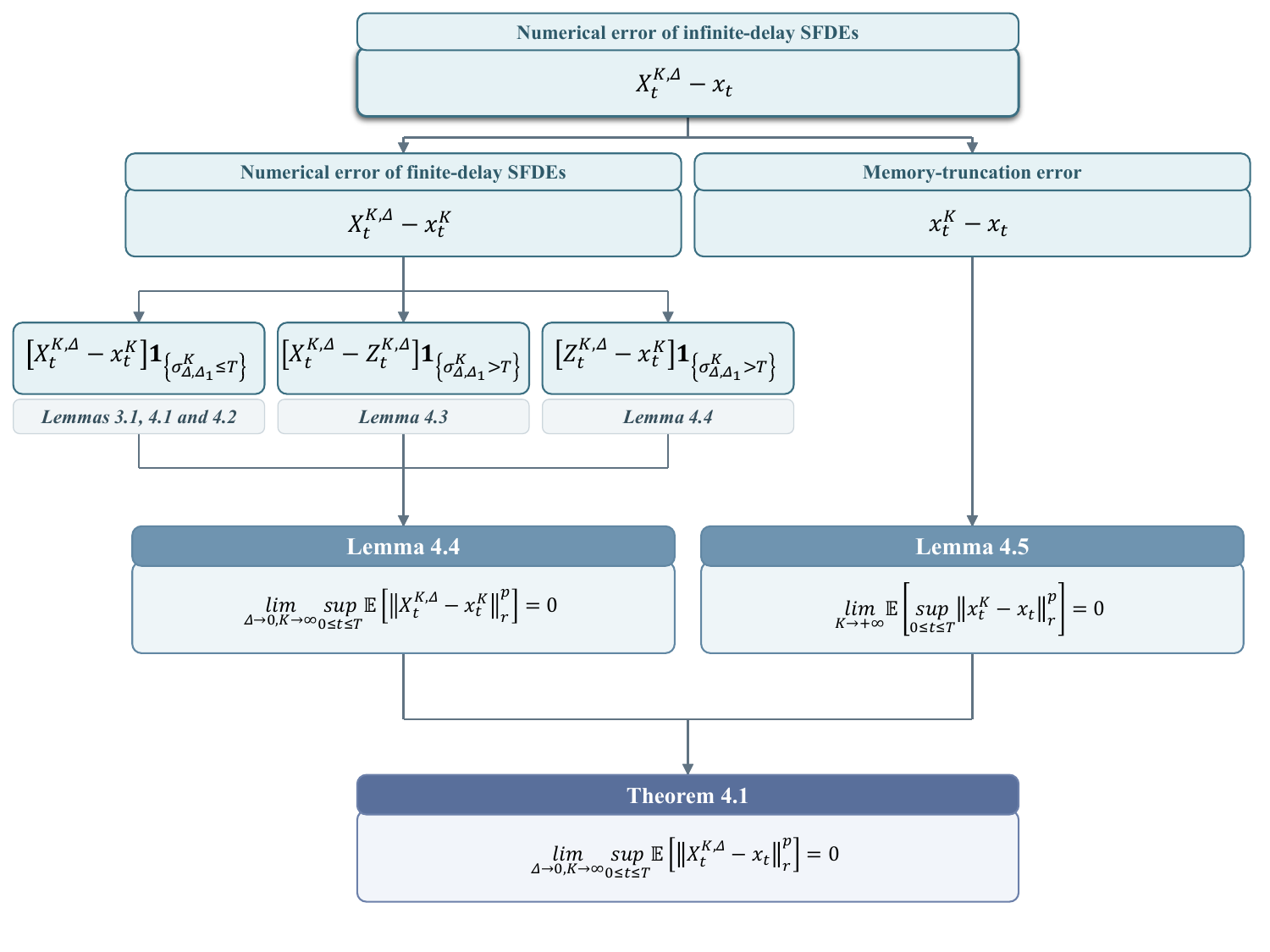}
\caption{Proof outline for strong convergence of the finite-memory TEM scheme}\label{3-tu}
\end{figure}
To this end, we impose the following assumption on the initial data.

\begin{assp}\label{xi1}
	The initial data $\xi \in \CC_r$ is uniformly continuous on $\mathbb R_-$, i.e., for any $s \ge 0$, 
	$ \mathcal{H}(s) := \sup_{|u_1-u_2| \le s;\  u_1,u_2 \in \RR_-} |\xi(u_1) - \xi(u_2)|  $
	is finite and $\lim_{s \to 0}\mathcal{H}(s) = 0.$
\end{assp}

Adapting the arguments in \cite[Lemma 4.2]{SWMW24} and \cite[(3.36) and Lemma 4.6]{LLM} to the weighted history norm, we obtain the following lemma using the uniform continuity of the initial data and the moment bounds from Lemma \ref{th3.1}, which are uniform in \(K\) and \(\Delta\). The proof is deferred to Appendix \ref{appA}.
\begin{lemma}\label{l3.6}
Let Assumptions \ref{a2.1-f}\text{-}\ref{a2.2} and \ref{xi1} hold.
Then for any $T > 0$, $\Delta \in (0, 1]$, $q > 0$, $\iota \in (0, q/2)$, and $\gamma > 0$, 
$$
\begin{aligned}
& \quad \sup_{K\in\mathbb N_+} \sup_{0 \le t \le T}  \EE \Big[ \sup_{u \le 0}
\left( e^{ru} | X^{K,\Delta}_t(u) -  Z^{K,\Delta}_t(u) | \right)^{q} \Big] \\
& \le  C_{T,\xi}  \left( \Delta^{\frac{q}{2}-\iota} + \mathcal{H}^q(\Delta)
+ \left( \Gamma^{-1}(L \Delta^{-\theta}) \right)^{-\gamma} \right),
\end{aligned}
$$
where $C_{T,\xi}$ and $\mathcal{H}(\cdot)$ are defined in Theorem \ref{th2.3} and Assumption \ref{xi1}, respectively.
\end{lemma}

\begin{lemma}\label{lemma3.6}
Let Assumptions \ref{a2.1-f}\text{-}\ref{a2.2} and \ref{xi1} hold. Then for any $T > 0$ and $q >0$,
	$$
	\lim_{\Delta \rightarrow 0,  K\to +\infty}  \left( \sup_{0 \le t \le  T}  \mathbb{E} \big[ \big\| X^{K,\Delta}_t - x^K_t\big\|_r^q \big] \right) = 0.
	$$
\end{lemma}
The proof of Lemma \ref{lemma3.6} is given in Appendix \ref{appA}.
We next estimate $\EE\big[\|x^K_t-x_t\|_r^q\big]$ under Assumptions \ref{a2.1-f}\text{-}\ref{a2.2} and \ref{xi1}.

\begin{lemma}\label{th3.2}
Let Assumptions \ref{a2.1-f}\text{-}\ref{a2.2} and \ref{xi1} hold. Then for any $T>0$ and $q > 0$,  
\begin{equation}
	\lim_{K \to \infty}\mathbb{E} \Big[ \sup_{0 \le t \le T}\| x^K_t - x_t\|_r^q \Big]= 0.
	\end{equation}
\end{lemma}

\begin{proof}
It suffices to consider \(q\ge2\), since the case \(0<q<2\) follows from H\"older’s inequality.
	Let $T > 0$ be arbitrary. For any $h > \|\xi\|_r$ and any integer $K \ge T$, let
	$\sigma^K_h := \tau_h \wedge \tau^K_h$, define
	$e^K(t) := x^K(t) - x(t), \,\forall\, t \in (-\infty, T]$
	and 
	$
	e^K_t := x^K_t - x_t, \, \forall\, t \in [0, T],
	$
	where $\tau_h$ and $\tau^K_h$ are defined in Theorem \ref{th2.3} and Lemma \ref{th2.3'}, respectively. Note that $e^K(u) = 0$ for all $u \le 0$.
	For any $\delta > 0$, $t \in [0, T]$, and $p > q$, by Young's inequality we obtain
	$$
	\begin{aligned}
	& \quad\,\, \EE \Big[\sup_{0 \le t \le T} \| e^K_t\|_r^q \Big] \\  
	& \le \EE \Big[ \sup_{0 \le t \le T} \| e^K_t\|_r^q \II_{\{ \sigma^K_h > T \}} \Big]  + \frac{q\delta}{p} \EE \Big[ \sup_{0 \le t \le T} \| e^K_t\|_r^p \Big]  +  \frac{p-q}{p\delta^{q/(p-q)} } \PP\{\sigma^K_h \le T\}.
	\end{aligned}
	$$
	It follows from the definition of $\|\cdot\|_r$ that
	\begin{equation}\label{3-3}
	\begin{aligned}
	\|e^K_t\|_r = \sup_{u \le 0} \left(e^{ru}|e^K(t+u)| \right) = \sup_{-\infty< u \le t} \left( e^{r(u-t)} |e^K(u)| \right)  
	\le \sup_{0 \le u \le t} |e^K(u)|.  
	\end{aligned}
	\end{equation}
By Assumption \ref{xi1}, $\mathcal H(w)\le(1+w)\mathcal H(1)$ for all $w\ge0$. This, together with \eqref{3-3}, yields that, for any $0 \leq t \leq T\le K$,
\begin{equation}\label{eq4.8}
\begin{aligned}
\|\pi_K(x_t^K)-x_t\|_r & \leq \|\pi_K(x_t^K)-x^K_t\|_r + \|x^K_t-x_t\|_r \\
& = \sup_{u\le-K} \left( e^{ru}|\xi(t-K)-\xi(t+u)| \right) + \|e^K_t\|_r\\
&\le e^{-rK}\sup_{w \geq 0} \left( e^{-rw}\mathcal H(w) \right) + \sup_{0 \leq u \leq t}|e^K(u)| \\
& \le C_\xi e^{-rK} + \sup_{0 \leq u \leq t}|e^K(u)|.
\end{aligned}
\end{equation}
Applying the H\"older inequality, the Burkholder-Davis-Gundy  inequalities and Assumptions \ref{a2.1-f} and \ref{a2.1-g}, we obtain
$$
\begin{aligned}
\EE \Big[\sup_{0\le s \le t}|e^K(s\wedge\sigma_h^K)|^q\Big] 
& \le C_{T,h} \EE \left[ \int_0^{t\wedge\sigma^K_h} \|\pi_K(x^K_s)-x_s\|_r^q \d s \right]\\
& \leq C_{T,h,\xi}e^{-qrK} + C_{T,h} \int_0^t \EE \left[ \sup_{0 \leq u \leq s}|e^K(u\wedge\sigma^K_h)|^q \right] \d s.
\end{aligned}
$$
An application of the Gronwall inequality implies that
$$
\EE\Big[\sup_{0\le t\le T}\|e_t^K\|_r^q\II_{\{\sigma_h^K>T\}}\Big]
\le\EE\Big[\sup_{0\le t\le T}|e^K(t\wedge\sigma_h^K)|^q\Big]
\le C_{T,h,\xi}e^{-qrK}.
$$
Therefore,
$$
	\EE \Big[ \sup_{0 \le t \le T} \| e^K_t\|_r^q \Big] 
	\le C_{T,h,\xi} e^{-qrK}  +  \frac{q\delta}{p} \EE \Big[ \sup_{0 \le t \le T} \| e^K_t\|_r^p \Big]  +  \frac{p-q}{p\delta^{q/(p-q)} } \PP\{\sigma^K_h \le T\}.
	$$
Let $\varepsilon > 0$ be arbitrary. By Theorem \ref{th2.3} and Lemma \ref{th2.3'}, we choose $\delta>0$ sufficiently small such that
	$
	\frac{q\delta}{p} 
	\EE \Big[ \sup_{0 \le t \le T} \| e^K_t\|_r^p \Big] \le C_{T,\xi} \delta \le \frac{\varepsilon}{3}.
	$
	We then choose $h$ sufficiently large such that
	$
	\frac{p-q}{p\delta^{q/(p-q)} } \PP\{\sigma^K_{h} \le T\} \le \frac{C_{T,\xi}}{\delta^{q/(p-q)} h^p} \le \frac{\varepsilon}{3}.
	$
Choosing $K \ge T$ sufficiently large, we obtain
$
	C_{T,h,\xi} e^{-qrK}  \le \frac{\varepsilon}{3}.
	$
This completes the proof.
\end{proof}

Lemmas \ref{lemma3.6} and \ref{th3.2} imply the $q$th-moment convergence of the finite-memory TEM numerical segment process.
\begin{theorem}\label{th3.4}
Let Assumptions \ref{a2.1-f}\text{-}\ref{a2.2} and \ref{xi1} hold. Then for any $T > 0$ and $q > 0$,
$$
	\lim_{\D \to 0, K \to +\infty}  \sup_{0 \le t \le T}  \mathbb{E}\big[ \big\| X^{K,\Delta}_t - x_t\big\|_r^q \big] = 0.
	$$
\end{theorem}

\subsection{Convergence rate}\label{sec3.3}
This subsection estimates the error between the finite-memory TEM numerical functional solution $X_t^{K,\Delta}$ and the exact functional solution $x_t$ of \eqref{ISFDE}. To this end, we first establish the exponential convergence of the functional solution $x_t^K$ of the finite-memory auxiliary equation \eqref{TSFDE} to $x_t$ as the memory truncation length $K\rightarrow +\infty$. We then bound the error between $x_t^K$ and $X_t^{K,\Delta}$, obtaining a temporal convergence order arbitrarily close to $1/2$ and exponential decay in $K$. To establish these results, we impose the following two assumptions, which are slightly stronger than those required for convergence alone.

\begin{assp}\label{a3.4}
There exist constants $\tilde p > 2$, $a_3 > 0$, $a_4 \geq 0$, and probability measures $\mu_2 \in \CP_{2 r}, \mu_3 \in \CP_{\tilde p r}$ such that for all $\phi,\varphi \in \CC_r$,
$$
\begin{aligned}
     \langle \phi(0)-\varphi(0), f(\phi)-f(\varphi) \rangle 
     & \le a_3  \int_{-\infty}^0  |\phi(u)-\varphi(u)|^2  \mu_2(\d u) -a_4 |\phi(0) - \varphi(0)|^{\tilde p} \\
     & \quad\,\, + a_4 \int_{-\infty}^0 |\phi(u) - \varphi(u)|^{\tilde p} \mu_3(\d u).
\end{aligned}
$$
\end{assp}

\begin{assp}\label{a3.5}
There exists a constant $a_5 > 0$ and a probability measure $\mu_4 \in \CP_{2r}$ such that for all $\phi, \varphi \in \mathcal{C}_{r}$,
$$
|g(\phi)-g(\varphi)|  \le a_5 \int_{-\infty}^0 |\phi(u)-\varphi(u)|  \mu_4(\d u).
$$
\end{assp}

As in the discussion following Assumption \(\ref{a2.2}\), all subsequent arguments involving Assumption \ref{a3.4} are presented only for the case \(a_4>0\). The case \(a_4=0\) follows analogously, without the conditions involving \(\mu_3\) and \(\widetilde p\).
Under Assumptions \ref{a2.1-f},  \ref{a3.4}, and \ref{a3.5}, Theorem \ref{th2.3} and Lemma \ref{th2.3'} remain valid, since Assumption \ref{a3.5} implies Assumption \ref{a2.1-g}, while Assumptions \ref{a3.4} and \ref{a3.5} together imply Assumption \ref{a2.2}.
Based on these results, we obtain the following exponential error estimate between $x_t^K$ and $x_t$.

\begin{lemma}\label{th3.6}
Let Assumptions \ref{a2.1-f}, \ref{a3.4}, and \ref{a3.5} hold, and suppose that there exist constants $c_1, c_2$ with $c_1  > 2$ and $c_2 > \tilde p$ such that $\mu_2, \mu_4 \in \CP_{c_1 r}$ and $\mu_3 \in \CP_{c_2 r}$. Then for any $K\in\mathbb N_+$, $T > 0$, and $q > 0$,
\begin{equation*}
\EE \Big[ \sup_{0 \le t \le T}  \|x^K_t-x_t\|_r^q  \Big] \le \tilde C_{T,\xi} e^{-\gamma_q rK},
\end{equation*}
where
$
\gamma_q
=
\frac{q}{2}\bigl((c_1-2)\wedge(c_2-\tilde p)\bigr)$, $
\tilde C_{T,\xi}
=
C_\xi\left(e^{CT^{(q/2)\vee 1}}+e^{CT^{(q\vee 2)\tilde p/4}}\right).
$
\end{lemma}

\begin{proof} 
Fix arbitrary $K\in\mathbb N_+$ and $T > 0$. Recall $e^K(t) = x^K(t) - x(t)$ defined in the proof of Lemma \ref{th3.2}. 
Applying It\^o's formula to $|e^K(t)|^2$ and using Assumptions \ref{a3.4}, \ref{a3.5}, we have that for any $t \in [0,T]$,
\begin{equation}\label{3-r1-1}
|e^K(t)|^2   
\le  M^K(t) + 2a_3 A^K_1(t) - 2 a_4 \int_0^t |e^K(s)|^{\tilde p} \d s + 2a_4 A^K_2(t) + a_5^2 A^K_3(t), 
\end{equation}
where
$
M^K(t) := 2 \int_0^{t}  \big\langle e^K(s), \left(g_K(x^K_s) - g(x_s)\right) \d B(s) \big\rangle
$
is a martingale, and
$$
\begin{gathered}
 A^K_1(t) := \int_0^t \int_{-\infty}^0  |x_s(u) - \pi_K(x^K_s)(u)|^2 \mu_2(\d u) \d s,\\
 A^K_2(t) := \int_0^t \int_{-\infty}^0  |x_s(u) - \pi_K(x^K_s)(u)|^{\tilde p} \mu_3(\d u) \d s, \\
 A^K_3(t) := \int_0^t \int_{-\infty}^0  |x_s(u) - \pi_K(x^K_s)(u)|^2 \mu_4(\d u) \d s.
\end{gathered}
$$
Using the definition of $\pi_K(\cdot)$, one has
\begin{equation}\label{3-r1-mu3}
\begin{aligned}
& A^K_1(t) 
 =   \int_0^t   \int_{-K}^0  |e^K_s(u)|^2  \mu_2(\d u)  \d s   +   \int_0^t    \int_{-\8}^{-K} |x_s(u) - x^K_s(-K)|^2 \mu_2(\d u)  \d s \\
 = & \int_{-K}^0  \int_0^t  |e^K(s+u)|^2   \d s   \mu_2(\d u) + \int_0^t    \int_{-\8}^{-K} |x_s(u) - x^K_s(-K)|^2 \mu_2(\d u)  \d s \\
 \le &  \int_0^t   |e^K(s)|^2    \d s  + 2  \mu_2^{(c_1 r)}  e^{-(c_1-2)rK}    \int_0^t  \left(   \|x_s\|_r^2 +   \|x^K_s\|_r^2  \right) \d s.
\end{aligned}
\end{equation}
Here, in the last inequality, the first term is estimated by the change of variables $v=s+u$ and the fact that $|e^K(v)| = 0$ for all $v \le 0$, and the second term is estimated using $\mu_2 \in \mathcal{P}_{c_1 r}$ and the argument in \cite[p.~168]{LLMS23}.
Similarly, it follows from $\mu_3\in\CP_{c_2 r}$ and $\mu_4 \in \CP_{c_1r}$ that
\begin{equation}\label{3-r1-mu45}
\begin{aligned}
& A^K_2(t) \le   \int_0^t   |e^K(s)|^{\tilde p}    \d s + 2^{\tilde p -1}  \mu_3^{(c_2 r)}  e^{-(c_2-\tilde p)rK} \int_0^t  \left(   \|x_s\|_r^{\tilde p} +   \|x^K_s\|_r^{\tilde p}  \right) \d s, \\
& A^K_3(t) \le \int_0^t   |e^K(s)|^2    \d s + 2  \mu_4^{(c_1 r)}  e^{-(c_1-2)rK} \int_0^t  \left(   \|x_s\|_r^2 +   \|x^K_s\|_r^2  \right) \d s.
\end{aligned}
\end{equation}
Inserting \eqref{3-r1-mu3} and \eqref{3-r1-mu45} into \eqref{3-r1-1} yields
\begin{equation*}
\begin{aligned}
       |e^K(t)|^2   
        &  \le  M^K(t)   +  C \int_0^t  |e^K(s)|^2   \d s    +  C  e^{-( (c_1-2) \wedge (c_2 - \tilde p) )rK}  \\
      & \quad \times \int_0^t   \left(  \|x_s\|_r^2  +   \|x^K_s\|_r^2 + \|x_s\|_r^{\tilde p} +   \|x^K_s\|_r^{\tilde p} \right) \d s.
\end{aligned}
\end{equation*}
For any $q \geq 2$, raising both sides to the power $q/2$ and using H\"older's inequality gives
\begin{align} \nonumber
    |e^K(t)|^q   
      &   \le 3^{\frac{q}{2}-1} |M^K(t)|^{\frac{q}{2}} +  C T^{\frac{q}{2}-1}  \int_0^{t}  |e^K(s)|^q   \d s   \\ \nonumber
      & \quad + C T^{\frac{q}{2}-1} e^{-\gamma_qrK} \int_0^{t}   \left(  \|x_s\|_r^q  +   \|x^K_s\|_r^q + \|x_s\|_r^{\frac{q \tilde p}{2}} +   \|x^K_s\|_r^{\frac{q \tilde p}{2}} \right)   \d s,
\end{align}
where $\gamma_q :=q\left( (c_1-2) \wedge (c_2 - \tilde p) \right) /2$. By Theorem \ref{th2.3} and Lemma \ref{th2.3'}, we arrive at
\begin{equation}\label{3-con1-2}
\begin{aligned}
   \EE \Big[ \sup_{0 \le s \le t} |e^K(s)|^q \Big]   
     &  \le 3^{\frac{q}{2}-1} \EE  \Big[ \sup_{0 \le s \le t} |M^K(s)|^{\frac{q}{2}} \Big]  \\
     & \quad +  C T^{\frac{q}{2}-1}  \int_0^t \EE  \Big[ \sup_{0 \le s \le u}  |e^K(s)|^q \Big]  \d u 
+  \tilde C_{T,\xi}  e^{-\gamma_qrK}.
\end{aligned}
\end{equation}
It follows from Assumption \ref{a3.5} and \eqref{3-r1-mu45} that
\begin{equation*}
\begin{aligned}
	         &\quad\int_0^{t}  | g(x_s) - g(\pi_K(x^K_s))  |^2 \d s \\
	      &\le  a_5^2 \int_0^{t} |e^K(s)|^2 \d s 
	      + 2a_5^2 \mu_4^{(c_1 r)} e^{-(c_1-2)rK}
	      \int_0^{t} \left(\|x_s\|_r^2 + \|x^K_s\|_r^2\right) \d s.
\end{aligned}
\end{equation*}
By the Burkholder--Davis--Gundy and Young inequalities, together with Theorem \ref{th2.3} and Lemma \ref{th2.3'}, we obtain
\begin{equation*}
\begin{aligned}
      & \quad   3^{\frac{q}{2}-1} \EE \Big[ \sup_{0 \le s \le t} |M^K(s)|^{\frac{q}{2}} \Big]    \\
      & \le \frac{1}{2}  \EE \Big[ \sup_{0 \le s \le t} |e^K(s)|^q \Big]  +   C  \EE \Big[ \Big( \int_0^{t}  | g(x_s) - g_K(x^K_s)  |^2 \d s\Big)^{\frac{q}{2}} \Big] \\
      & \le \frac{1}{2}  \EE \Big[ \sup_{0 \le s \le t} |e^K(s)|^q \Big]   +    C T^{\frac{q}{2}-1} \int_0^t \EE \Big[ \sup_{0 \le s \le u} |e^K(s)|^q\Big] \d u
      +  \tilde C_{T,\xi}  e^{-\gamma_q rK}. 
\end{aligned}
\end{equation*}
Inserting the above inequality into \eqref{3-con1-2} gives
$$
      \EE \Big[ \sup_{0 \le s \le t}   |e^K(s)|^q \Big]
      \le  C T^{\frac{q}{2}-1} \int_0^t \EE \Big[\sup_{0 \le s \le u} |e^K(s)|^q \Big]\d u 
      + \tilde C_{T,\xi} e^{-\gamma_qrK}.
$$
The Gronwall inequality and \eqref{3-3} imply that
$
    \EE \big[ \sup_{0 \le t \le T}  \|e^K_t\|_r^q \big] \le  \tilde C_{T,\xi} e^{-\gamma_q rK}.
$
For $0<q<2$, the desired estimate follows directly from H\"older's inequality.
\end{proof}

To estimate the error between $X^{K,\Delta}_t$ and $x^K_t$, we strengthen the uniform continuity condition on the initial data in Assumption \ref{xi1} to the following Hölder continuity condition, and impose a polynomially growing Lipschitz condition on the drift coefficient.
\begin{assp}\label{a5.8}
There exist constants $a_6>0$ and $a_7 \in [1/2,1]$ such that the initial data $\xi$ satisfies
	$$
	|\xi(t_1)-\xi(t_2)|\le a_6 |t_1-t_2|^{a_7},\quad \forall \, t_1, t_2 \in (-\8, 0].
	$$
\end{assp}

\begin{assp}\label{a3.3}
	There exist constants $a_8 > 0$ and $v > 0$ such that for any $\p, \f\in \CC_r$,
	$$
	|f(\phi)-f(\varphi)|
	\le  a_8 \|\phi-\varphi\|_r  \left( 1+ \|\phi\|_r^v+ \|\varphi\|_r^v   \right).
	$$
\end{assp}

We specify a suitable choice of the truncation mapping $\Pi^\Delta(\cdot)$
used in the construction of the finite-memory TEM scheme under
Assumption~\ref{a3.3}, as stated in the following remark.

\begin{rmk}\label{rmk3.12}
Under Assumption \ref{a3.3}, the function $\Gamma$ in \eqref{trun} may be taken as
$
\Gamma(R)=  a_8  \left( 1 + 2 R^v  \right), \, \forall \, R \ge 0.
$
For each $q>0$, let $p$ be any constant satisfying $p>q(v+1)$, and set $
\theta=\frac{qv}{2(p-q)}.$
Since $p-q>qv>0$, it follows that $\theta\in(0,1/2)$.
Fix $L > 1 \vee |f(\mathbf 0)| \vee a_8$. For any $\Delta\in(0,1]$, we have $L\Delta^{-\theta}>a_8$, and hence
$
\Gamma^{-1}(L \Delta^{-\theta})=\left(\frac{L \Delta^{-\theta}}{2 a_8} - \frac{1}{2 }\right)^{1/v},\quad \forall \, \Delta \in (0, 1].
$
Therefore, the truncation mapping is given by 
$
\Pi^\Delta(x) = \bigg( |x|\wedge\left(\frac{L \Delta^{-\theta}}{2 a_8} - \frac{1}{2 }\right)^{1/v} \bigg)\frac{x}{|x|}, \quad \forall x \in \RR^n.
$
\end{rmk}

Under Assumptions \ref{a5.8} and \ref{a3.3}, the estimate in Lemma \ref{l3.6} admits a more refined form, as stated in the following lemma.
\begin{lemma}\label{l3.12}
Let Assumptions \ref{a2.1-f}\text{-}\ref{a2.2}, \ref{a5.8} and \ref{a3.3} hold. Then for any $T > 0$, $\Delta \in (0,1]$, $q > 0$, and \(\iota\in(0,q/2)\),
$$
\begin{aligned}
& \sup_{K\in\mathbb N_+} \sup_{0 \le t \le T}  \EE \Big[\sup_{u \le 0} \Big( e^{ru} | Z^{K,\Delta}_t(u) -  X^{K,\Delta}_t(u) | \Big)^{q} \Big]
\le  C_{T,\xi}  \Delta^{\frac{q}{2}-\iota},
\end{aligned}
$$
where $C_{T,\xi}$is defined in Theorem  \ref{th2.3}.
\end{lemma}
\begin{proof}
For any $\Delta \in (0,1]$, under Assumption \ref{a5.8}, the function $\mathcal H(\cdot)$ given in Assumption \ref{xi1} satisfies $\mathcal H(\Delta) \le a_6 \Delta^{1/2}$. Applying Lemma~\ref{l3.6} with $\gamma=p-q$ and using Remark~\ref{rmk3.12} with
$\theta=\frac{qv}{2(p-q)}$, we derive the desired assertion.
\end{proof}

\begin{lemma}\label{l3.13}
Let Assumptions \ref{a3.4}\text{-}\ref{a3.3} hold. Then for any $T>0$, $\Delta \in (0,1]$, $q > 0$, and $\kappa \in (0, q/2)$,
$$
	\sup_{K\in\mathbb N_+} \EE \Big[ \sup_{0 \le t \le T} |x^K(t) - Z^{K,\Delta}(t)|^q \II_{\{\vartheta^K_\Delta > T\}}  \Big] \le C_{T,\xi}   \Delta^{\frac{q}{2}-\kappa},
$$
where 
$
\vartheta^K_\Delta := \tau^K_{\Gamma^{-1}(L \Delta^{-\theta})}  \wedge  \tau^K_{\Delta,\Delta},
$
$C_{T,\xi}$, $\tau^K_{\Gamma^{-1}(L \Delta^{-\theta})}$, and $\tau^K_{\Delta,\Delta}$ are defined in Theorem \ref{th2.3}, Lemma \ref{th2.3'}, and \eqref{rh}, respectively.
\end{lemma}

\begin{proof}
It suffices to prove the result for $q>2$, since the case $ q\in(0,2]$ follows from H\"older's inequality. Let $T > 0$ and $\Delta \in (0,1]$ be arbitrary. 
For any integer $K\in\mathbb N_+$, define $
e^{K,\Delta}(t):=x^K(t)-Z^{K,\Delta}(t).
$ Then,
$$
\begin{aligned}
e^{K,\Delta}(t \wedge \vartheta^K_\Delta) & = e^{K,\Delta}(0) +  \int_0^{t \wedge \vartheta^K_\Delta} \left( f_K(x^K_s) - f(X^{K,\Delta}_s) \right)  \d s \\
& \quad + \int_0^{t \wedge \vartheta^K_\Delta} \left( g_K(x^K_s) - g(X^{K,\Delta}_s) \right)  \d B(s).
\end{aligned}
$$
By It${\rm \hat{o}}$'s formula and the inequality $|a+b|^2 \le 2|a|^2 + 2|b|^2$, we have
\begin{equation}\label{3-e}
\begin{aligned}
  & \quad |e^{K,\Delta}(t \wedge \vartheta^K_\Delta)|^q  \le |e^{K,\Delta}(0)|^q + I_1(t \wedge \vartheta^K_\Delta) + q \int_0^{t \wedge \vartheta^K_\Delta} \Big[ |e^{K,\Delta}(s)|^{q-2}  \\
  & \quad \times \big( \langle e^{K,\Delta}(s), f_K(x^K_s) - f(X^{K,\Delta}_s) \rangle  + \frac{(q-1)}{2}|g_K(x^K_s) - g(X^{K,\Delta}_s)|^2 \big) \Big] \d s \\
&  \le |e^{K,\Delta}(0)|^q + I_1(t \wedge \vartheta^K_\Delta)  + I_2(t \wedge \vartheta^K_\Delta) + I_3(t \wedge \vartheta^K_\Delta) ,
\end{aligned}
\end{equation}
where
\begin{align}\nonumber
I_1(t \wedge \vartheta^K_\Delta)
& :=q\int_0^{t \wedge \vartheta^K_\Delta}
|e^{K,\Delta}(s)|^{q-2}
\left\langle e^{K,\Delta}(s),
\left(g_K(x^K_s)-g(X^{K,\Delta}_s)\right)\d B(s)\right\rangle,\\ \nonumber
I_2(t \wedge \vartheta^K_\Delta)
&:=q\int_0^{t \wedge \vartheta^K_\Delta} |e^{K,\Delta}(s)|^{q-2}
\Big(\left\langle e^{K,\Delta}(s),f_K(x^K_s)-f(Z^{K,\Delta}_s)\right\rangle\\ \nonumber
&\quad +(q-1)|g_K(x^K_s)-g(Z^{K,\Delta}_s)|^2\Big)\d s,\\ \nonumber
I_3(t \wedge \vartheta^K_\Delta)
&:=q\int_0^{t \wedge \vartheta^K_\Delta} |e^{K,\Delta}(s)|^{q-2}
\Big(\left\langle e^{K,\Delta}(s),f(Z^{K,\Delta}_s)-f(X^{K,\Delta}_s)\right\rangle\\ \nonumber
&\quad +(q-1)|g(Z^{K,\Delta}_s)-g(X^{K,\Delta}_s)|^2\Big)\d s.
\end{align}
Adopting the arguments in the proof of \cite[(4.15)]{LLM}, we obtain 
\begin{equation}\label{I_2}
I_2(t \wedge \vartheta^K_\Delta) \le C  \int_0^{t \wedge \vartheta^K_\Delta} |e^{K,\Delta}(s)|^q \d s.
\end{equation}
Furthermore, by Assumptions \ref{a3.5}, \ref{a3.3}, and the Young inequality, we deduce
\begin{equation}\label{I_3}
	\begin{aligned}
& \quad  I_3(t \wedge \vartheta^K_\Delta)  \\
& \le C  \int_0^{t \wedge \vartheta^K_\Delta}   \left( |e^{K,\Delta}(s)|^q 
+   | f(Z^{K,\Delta}_s) - f(X^{K,\Delta}_s) |^q   +    | g(Z^{K,\Delta}_s) - g(X^{K,\Delta}_s)  |^q \right) \d s \\
&  \le C \int_0^{t \wedge \vartheta^K_\Delta}    |e^{K,\Delta}(s)|^q \d s
+ C I_{3,0}(t \wedge \vartheta^K_\Delta), 
\end{aligned}
\end{equation}
where
	$
	I_{3,0}(t \wedge \vartheta^K_\Delta) = \int_0^{t \wedge \vartheta^K_\Delta}   \big\| Z^{K,\Delta}_s - X^{K,\Delta}_s \big\|^q_r \big( 1  +   \big\| Z^{K,\Delta}_s \big\|^{qv}_r + \big\|X^{K,\Delta}_s \big\|_r^{qv} \big) \d s.
	$
	Inserting \eqref{I_2} and \eqref{I_3} into \eqref{3-e}, one arrives at
\begin{equation*}
	|e^{K,\Delta}(t \wedge \vartheta^K_\Delta )|^q
	\le  |e^{K,\Delta}(0)|^q +  I_1(t \wedge \vartheta^K_\Delta)   +  C  \int_0^{t \wedge \vartheta^K_\Delta}    |e^{K,\Delta}(s)|^q \d s + C I_{3,0}(t \wedge \vartheta^K_\Delta).
\end{equation*}
Therefore, 
\begin{equation*}
\begin{aligned}
\EE \Big[ \sup_{0 \le s \le t} |e^{K,\Delta}(s \wedge \vartheta^K_\Delta)|^q  \Big] 
& \le  |e^{K,\Delta}(0)|^q + \EE \Big[ \sup_{0 \le s \le t}  |I_1(s \wedge \vartheta^K_\Delta )|   \Big] \\
& \quad + C \int_0^{t} \EE [ |e^{K,\Delta}(s \wedge \vartheta^K_\Delta)|^q ] \d s  + C \EE [I_{3,0}(t \wedge \vartheta^K_\Delta) ].
\end{aligned}
\end{equation*}
By the definition of the truncation mapping and Remark \ref{rmk3.12}, with $p$ and $\theta$ as specified there,
$|e^{K,\Delta}(0)|^q=|\xi(0)-\Pi^\Delta(\xi(0))|^q
\le |\xi(0)|^p\bigl(\Gamma^{-1}(L\Delta^{-\theta})\bigr)^{-(p-q)}
\le C_\xi\Delta^{q/2}.$ Moreover, by H\"older's inequality, Lemmas \ref{th3.1}, \ref{l3.5}, and \ref{l3.12}, we obtain that for any $\kappa \in (0, q/2)$,
	\begin{align*}
 &\quad\EE [I_{3,0}(t \wedge \vartheta^K_\Delta)]  \le  C \int_0^t  \left\{ \EE \Big[ \sup_{u \le 0} \left( e^{ru} \big| Z^{K,\Delta}_s(u) - X^{K,\Delta}_s(u) \big| \right)^{q(1+v)} \Big] \right\}^{\frac{1}{v+1}} \\
	& \quad \times  \left\{ \EE \Big[ 1 + \|\xi\|_r^{q(1+v)} +  \sup_{0 \le u \le s} \big| Z^{K,\Delta}(u) \big|^{q(1+v)} + \big\|X^{K,\Delta}_s \big\|_r^{q(1+v)} \Big] \right\}^{\frac{v}{v+1}} \d s \\
	& \le  C_{T,\xi}  \Delta^{\frac{q}{2}-\kappa}.
	\end{align*} 
Combining this estimate with the bounds for $|e^{K,\Delta}(0)|^q$ and $\mathbb E[I_{3,0}(t \wedge \vartheta^K_\Delta)]$, we obtain 
	\begin{equation*}
	\begin{aligned}
	 &\quad\EE  \Big[ \sup_{0 \le s \le t} |e^{K,\Delta}(s \wedge \vartheta^K_\Delta)|^q \Big] \\
	& \le \EE \Big[ \sup_{0 \le s \le t}  |I_1(s \wedge \vartheta^K_\Delta )|  \Big]  +   C  \int_0^{t} \EE[ |e^{K,\Delta}(s \wedge \vartheta^K_\Delta)|^q] \d s  +  C_{T,\xi}  \Delta^{\frac{q}{2}-\kappa}.
	\end{aligned}
	\end{equation*}
By the Burkholder-Davis-Gundy and Young inequalities, together with arguments analogous to those used in the estimates of $I_2$ and $I_3$, we further obtain
	\begin{equation*}
	\begin{aligned}
	& \quad \EE \Big[ \sup_{0 \le s \le t} |I_1(s \wedge \vartheta^K_\Delta )|  \Big] \\
		& \le  \frac{1}{2}  \EE \Big[ \sup_{0 \le s \le t} |e^{K,\Delta}(s \wedge \vartheta^K_\Delta )|^q  \Big] + C \int_0^{t} \EE [ |e^{K,\Delta}(s \wedge \vartheta^K_\Delta)|^q] \d s  + C_{T,\xi} \Delta^{\frac{q}{2}-\kappa}.
	\end{aligned}
	\end{equation*}
Consequently, 
	$$
	\EE  \Big[ \sup_{0 \le s \le t} |e^{K,\Delta}(s \wedge \vartheta^K_\Delta)|^q \Big]  
	 \le  C  \int_0^t  \EE \Big[ \sup_{0 \le s \le u} \big| e^{K,\Delta}(s \wedge \vartheta^K_\Delta) \big|^q  \Big] \d u +  C_{T,\xi} \Delta^{\frac{q}{2}-\kappa}.
	$$
An application of the Gronwall inequality gives
$$
	\EE \big[ \sup_{0 \le s \le T} |e^{K,\Delta}(s)|^q \II_{\{\vartheta^K_\Delta > T\}} \big] \le \EE \big[ \sup_{0 \le s \le T} |e^{K,\Delta}(s \wedge \vartheta^K_\Delta)|^q\big] \le  C_{T,\xi}  \Delta^{\frac{q}{2}-\kappa}.
$$
Since $C_{T,\xi}$ is independent of $K$ and $\Delta$, taking
 the supremum over $K$ yields the desired assertion.
\end{proof}

\begin{lemma}\label{th3.16}
Let Assumptions \ref{a3.4}-\ref{a3.3} hold. Then for any $K\in\mathbb N_+$, $T > 0$, $\Delta \in (0,1]$, $q>0$, and $\kappa \in (0, q/2)$,
$$
\sup_{0 \le t \le T} \EE \big[ \|X^{K,\D}_t-x^K_t\|_r^q \big]  \le C_{T,\xi}  \left(  \Delta^{\frac{q}{2}-\kappa} + e^{-qrK}\right),
$$
where $C_{T,\xi}$ is defined in Theorem \ref{th2.3}.
\end{lemma}

\begin{proof}
Fix $K\in\mathbb N_+$, $T > 0$ and $\Delta \in (0,1]$ arbitrarily. For any $q > 2$ and $t \in [0,T]$, we have
	\begin{equation}\label{e-kD}
	\begin{aligned}
& \quad \EE \big[ \|X^{K,\D}_t-x^K_t\|_r^q \big] \\  
	&\le  \EE \big[ \|X^{K,\D}_t-x^K_t\|_r^q  \II_{\{\vartheta^K_\Delta \le T\}}  \big]  + 2^q  \EE \big[  \|X^{K,\D}_t - Z^{K,\D}_t\|_r^q \II_{\{\vartheta^K_\Delta > T\}} \big]    \\
	& \quad + 2^q  \EE \big[ \| Z^{K,\D}_t-x^K_t\|_r^q \II_{\{\vartheta^K_\Delta > T\}}  \big] ,  
	\end{aligned}
	\end{equation}
where $\vartheta^K_\Delta$ is defined in Lemma \ref{l3.13}.
By Young's inequality, Lemmas \ref{th3.1}, \ref{th2.3'}, and \ref{l3.5}, Remark \ref{rmk3.12}, we have
\begin{equation}\label{vartheta}
\begin{aligned}
& \quad\,\, \mathbb{E} \big[ \big\| X^{K,\Delta}_t - x^{K}_t\big\|_r^q \mathbf{1}_{\{\vartheta^K_\Delta \le T\}} \big] \\
& \le \frac{q \Delta^{q/2}}{p}  \mathbb{E} \big[ \big\| X^{K,\Delta}_t - x^{K}_t\big\|_r^p \big] + \frac{p-q}{p \Delta^{q^2/[2(p-q)]}}  \mathbb{P}{\{\vartheta^K_\Delta \le T\}} \\
& \le \frac{2^{p-1} q \Delta^{q/2}}{p} \left( \EE \big[ \big\| X^{K,\Delta}_t\big\|_r^p \big] + \EE \big[ \big\| x^{K}_t\big\|_r^p \big] \right) \\
& \quad\, + \frac{p-q}{p \Delta^{q^2/[2(p-q)]}} \Big( \PP \{  \tau^K_{\Gamma^{-1}(L \Delta^{-\theta})}\le T  \} + \PP \{ \tau^{K}_{\Delta,\Delta}  \le T \} \Big) \\
& \le C_{T,\xi}\Delta^{q/2} .
\end{aligned}
\end{equation}
Applying Lemma \ref{l3.12}, we have that for any $\kappa \in (0, q/2)$, 
	\begin{equation}\label{3-3-3.50}
	\begin{aligned}
\EE \big[  \|X^{K,\D}_t-Z^{K,\D}_t\|_r^q \II_{\{\vartheta^K_\Delta > T\}} \big] 
    & \le \EE \Big[ \sup_{u \le 0} \left( e^{ru} |X^{K,\D}_t(u)-Z^{K,\D}_t(u)| \right)^q \Big] \\
    & \le C_{T,\xi}  \Delta^{\frac{q}{2}-\kappa}.
	\end{aligned}
	\end{equation}
	Substituting \eqref{vartheta} and \eqref{3-3-3.50} into \eqref{e-kD} implies
	\begin{equation}\label{e-kD'}
	\begin{aligned}
	\EE \big[ \|X^{K,\D}_t-x^K_t\|_r^q \big]
	\le C_{T,\xi}  \Delta^{\frac{q}{2}-\kappa}  +  2^q  \EE \big[ \|Z^{K,\D}_t-x^K_t\|_r^q \II_{\{\vartheta^K_\Delta > T\}}  \big].
	\end{aligned}
	\end{equation}
It follows from Lemma \ref{l3.13} that
\begin{align}\nonumber
& \quad  \EE \big[ \|Z^{K,\D}_t-x^K_t\|_r^q \II_{\{\vartheta^K_\Delta > T\}}  \big]   \\ \nonumber
& \le \EE \Big[ \sup_{0 \le u \le t}  |e^{K,\D}(u)|^q \II_{\{\vartheta^K_\Delta > T\}}  \Big] + \EE \Big[ \sup_{u \le -K} \left( e^{ru}   |Z^{K,\D}_t(-K)-x^K_t(u)| \right)^q \Big]   \\ \nonumber
&\le C_{T,\xi}  \Delta^{\frac{q}{2}-\kappa} + \EE\Big[ \sup_{u \le -K} \left( e^{ru}   |Z^{K,\D}_t(-K)-x^K_t(u)| \right)^q  \Big] \II_{[K,+\infty)}(t) \\ \nonumber
& \quad +  \sup_{u \le -K} \left( e^{ru}|\xi(t-K)-\xi(t+u)| \right)^q \II_{[0,K)}(t),
\end{align}
where $e^{K,\Delta}(t) = x^K(t)-Z^{K,\Delta}(t)$ is defined in the proof of Lemma \ref{l3.13}.
By Lemma \ref{th2.3'} and Lemma \ref{l3.5}, one has that, for any $t \geq K$,
$$
\begin{aligned}
& \quad  \EE \Big[ \sup_{u \le -K} \Big( e^{ru}   |Z^{K,\D}_t(-K)-x^K_t(u)| \Big)^q \Big] \\
&\le  2^q \EE \Big[ \sup_{u \le -K} \left( e^{ru} |Z^{K,\D}_t(-K)| \right)^q +  \sup_{u \le -K} \left(e^{ru}   |x^K_t(u)| \right)^q  \Big]  \\
& \le  2^q \EE\Big[\Big(   \sup_{0 \le u \le t}  |Z^{K,\D}(u)|^q +  \sup_{0 \le u \le t} |x^K(u)|^q \Big)e^{-qrK}  + \sup_{u \le 0} \left(e^{r(u-t)}   |x^K(u)| \right)^q \Big] \\
&\le  2^q \|\xi\|_r^q e^{-qrK}    +  2^q  \Big( \EE  \Big[ \sup_{0 \le u \le t}  |Z^{K,\D}(u)|^q \Big] + \EE \Big[  \sup_{0 \le u \le t} |x^K(u)|^q \Big] \Big)e^{-qrK}   \\
& \le  C_{T,\xi}  e^{-qrK}.
\end{aligned}
$$
Arguing as in the derivation of \eqref{eq4.8} and using Assumption \ref{a5.8}, we obtain
$\sup_{u\le-K}\left(e^{ru}|\xi(t-K)-\xi(t+u)|\right)^q \le C_\xi e^{-qrK}
$ for any $0 \le t < K$.
Therefore, 
$
\EE \big[ \|Z^{K,\D}_t-x^K_t\|_r^q \II_{\{\vartheta^K_\Delta > T\}}  \big] 
\le   C_{T,\xi} \left(  \Delta^{\frac{q}{2}-\kappa} + e^{-qrK}\right).  
$
This, together with \eqref{e-kD'}, yields
	$
	\EE [\|X^{K,\D}_t-x^K_t\|_r^q]  \le C_{T,\xi}  (  \Delta^{\frac{q}{2}-\kappa} + e^{-qrK}).
	$
	The proof is completed.
\end{proof}

Combining Lemmas \ref{th3.6} and \ref{th3.16}, we obtain the convergence rate of the finite-memory TEM numerical scheme \eqref{TEM}.

\begin{theorem}\label{th3.20}
Let Assumptions \ref{a3.4}\text{-}\ref{a3.3} hold, and assume that  $\mu_2, \mu_4 \in \CP_{c_1 r}$, $\mu_3 \in \CP_{c_2 r}$, with $c_1$ and $c_2$ as in
Lemma~\ref{th3.6}. Then for any $K\in\mathbb N_+$, $T > 0$, $\Delta \in (0,1]$, $q > 0$, and $\varepsilon \in \left(0, q/2\right)$,
\begin{equation*}
\sup_{0 \le t \le T}  \EE \big[ \|X^{K,\Delta}_t - x_t\|_r^q  \big] \le C_{T,\xi} \left( \Delta^{\frac{q}{2}-\varepsilon} +  e^{-\alpha_q K} \right),
\end{equation*}
where $\alpha_q = (\gamma_q r)\wedge qr$, $\gamma_q$ is defined Lemma \ref{th3.6}, respectively.
\end{theorem}

In particular, for a fixed auxiliary parameter $M\in\mathbb N_+$, independent of $\Delta$, define
\begin{equation}\label{k_De}
K_{\D} = \left\lceil \frac{-q\ln(\D/M)}{2\alpha_q} \right\rceil \vee 1.
\end{equation}
The parameter $M$ provides additional flexibility in the choice of the memory length.
For notational simplicity, we suppress the dependence of $K_{\Delta}$ on $M$. Define $X^\Delta_t:=X^{K_{\Delta},\Delta}_t,\,t\geq0$,
to be the corresponding finite-memory TEM numerical functional solution.
It then follows from Theorem \ref{th3.20} that $X^\Delta_t$ converges to the exact segment $x_t$  as $\Delta \to 0$, at a rate arbitrarily close to $1/2$.

\begin{coro}\label{rmk3.19} Let the conditions in Theorem \ref{th3.20} hold. Then for any $T > 0$, $q>0$, $\Delta \in (0,1]$, and $\varepsilon \in (0,q/2)$,
$$
\sup_{0 \le t \le T}  \EE \big[\big\|X^\D_t -x_t \big\|_r^q\big]   \le  C_{T,\xi}  \Delta^{\frac{q}{2}-\varepsilon}.
$$
\end{coro}

\section{Approximation of invariant measures}\label{S4}
This section investigates the numerical approximation of the IPM associated with the infinite-delay equation \eqref{ISFDE}. Under suitable dissipativity conditions, we first establish the existence and uniqueness of the IPM for the exact segment process of \eqref{ISFDE} and of the corresponding numerical IPM for the finite-memory TEM segment process. 
Combining their ergodic properties with the finite-time strong convergence of the numerical segment process, 
we then prove that the numerical IPM converges to the exact IPM in the $L^2$-Wasserstein distance. 
Furthermore, based on the strong convergence rate of the numerical segment process, we derive the convergence rate for the numerical IPM.

We begin by introducing the necessary notation and imposing a dissipativity condition on the drift coefficient.
Let $\BB(\CC_r)$ denote the Borel $\sigma$-algebra on $\CC_r$, and let $\CP(\CC_r)$ denote the set of all probability measures on $(\CC_r, \BB(\CC_r))$. For $l \ge 1$, define
$$
\CP_l(\CC_r) := \bigg\{ \m \in \CP(\CC_r) :  \int_{\CC_r} \|\p\|_r^l \m(\d \p) < \infty \bigg\}.
$$
For any $\hat\nu_1, \hat\nu_2 \in \CP_l(\CC_r)$, their $l$-Wasserstein distance is defined by
$$
\WW_l(\hat\nu_1, \hat\nu_2) = \Big(  \inf_{\tilde\nu\in \Pi(\hat\nu_1,\hat\nu_2)} \int_{\CC_r \K \CC_r} \|\p-\f\|_r^l \tilde\nu (\d \p, \d \f) \Big)^{1/l},
 $$
where $\Pi(\hat\nu_1,\hat\nu_2)$ is the set of all probability measures on $\CC_r \K \CC_r$ with marginal measures  $\hat\nu_1$ and $\hat\nu_2$. 
 Since $\CC_r$ is a Polish space, $(\CP_l(\CC_r),\WW_l)$ is also a Polish space; see \cite[Theorem 6.18]{V09}.
To emphasize dependence on the initial datum, denote by
$\{x^\xi(t)\}_{t\in\mathbb{R}}$ the exact solution process of \eqref{ISFDE} with the initial datum $\xi\in\mathcal{C}_r$, and by
$\{x_t^\xi\}_{t\geq0}$ its associated segment process.
For any $\xi \in \mathcal{C}_r$ and $t \geq 0$, let $\pi_t^\xi$ denote the probability measure induced
by $x_t^\xi$ on $\mathcal{C}_r$, namely,
\[
\pi_t^\xi(A)
=
\mathbb{P}\{\omega \in \Omega : x_t^\xi(\omega) \in A\},
\qquad
A \in \mathcal{B}(\mathcal{C}_r).
\]
Define the linear operator
$P_t : \mathcal{B}_b(\mathcal{C}_r) \to \mathcal{B}_b(\mathcal{C}_r)$
induced by $x_t^\xi$ as
\[
P_t \Psi(\xi)
:=
\mathbb{E}[\Psi(x_t^\xi)]
=
\int_{\mathcal{C}_r} \Psi(\phi)\,\pi_t^\xi(\mathrm{d}\phi),
\]
where $\mathcal{B}_b(\mathcal{C}_r)$ denotes the set of all bounded
Borel-measurable functions on $\mathcal{C}_r$.
By \cite[Theorem 4.2]{WYM17}, the segment process
$\{x_t^\xi\}_{t\geq0}$ is a time-homogeneous  Markov process.
Consequently, $\{P_t\}_{t\geq0}$ is a Markov semigroup.

\begin{assp}\label{dis}
There exist constants $b_1>0,\, b_2>0$, and a probability measure $\mu_5 \in \CP_{2r}$ such that for any $\phi, \varphi \in \CC_r$,
$$
 \langle \p(0)-\f(0), f(\p)-f(\f) \rangle 
\le  -b_1 |\p(0)-\f(0)|^2  +  b_2  \int_{-\8}^0  |\p(u)-\f(u)|^2 \mu_5(\d u).
$$
\end{assp}

It follows from Assumptions \ref{a3.5}, \ref{dis}, and the Young inequality that for any $\phi\in\CC_r$ and $\rho > 0$,
\begin{align}
|g(\phi)|^2
&\leq
(1+\frac{1}{\rho})|g(\mathbf{0})|^2
+(1+\rho)a_5^2\int_{-\infty}^{0}|\phi(u)|^2\,\mu_4(\mathrm{d}u),
\label{mon-g}\\
\langle\phi(0),f(\phi)\rangle
&\leq
\frac{1}{2\rho}|f(\mathbf{0})|^2
-\left(b_1-\frac{\rho}{2}\right)|\phi(0)|^2
+b_2\int_{-\infty}^{0}|\phi(u)|^2\,\mu_5(\mathrm{d}u).
\label{mon-f}
\end{align}
In addition, if $\mu_4, \mu_5\in \CP_{p r}$ for some $p \geq 2$, a further application of the Young inequality yields that for any $\phi\in\CC_r$,
\begin{equation}\label{4-mon}
\begin{aligned}
& \quad |\phi(0)|^{p-2} \left(  2 \langle \phi(0),  f(\phi) \rangle + (p-1)|g(\phi)|^2  \right)  \\
& \le  K_1 - \alpha_1 |\p(0)|^p + \alpha_2 \int_{-\infty}^0 |\phi(u)|^p \mu_5(\mathrm{d} u) + \alpha_3 \int_{-\infty}^0 |\phi(u)|^p \mu_4(\d u),
\end{aligned}
\end{equation}
and for any $\phi,\varphi\in\CC_r$,
\begin{equation}\label{attr+}
\begin{aligned}
    & \quad |\phi(0) - \varphi(0)|^{p-2} \left(  2 \langle \phi(0)-\varphi(0), f(\phi)-f(\varphi)\rangle   + (p-1)|g(\phi)-g(\varphi)|^2  \right)  \\
  & \le - \beta_1 |\phi(0)-\varphi(0)|^p
  + \beta_2 \int_{-\infty}^0 |\phi(u)-\varphi(u)|^p \mu_5(\d u) \\
  & \quad + \beta_3 \int_{-\infty}^0 |\phi(u)-\varphi(u)|^p \mu_4(\d u),
\end{aligned}
\end{equation}
where 
\begin{align*}
&K_1 := \frac{2}{p} \rho^{-\frac{p-2}{2}} \left( \frac{1}{\rho}|f(\mathbf{0})|^2 + (p-1)(1+\frac{1}{\rho})|g(\mathbf{0})|^2 \right)^{\frac{p}{2}},\\
&\alpha_1 := 2b_1 - \frac{2(p-1)\rho}{p} - \frac{p-2}{p} \Big( 2b_2 + (p-1) (1+\rho)a_5^2\Big), \\
&\alpha_2 := \frac{4b_2}{p}, \quad \alpha_3 := \frac{2(p-1)(1+\rho)a_5^2}{p},\\
&\beta_1 := 2b_1 - \frac{p-2}{p} \Big( 2b_2 + (p-1)a_5^2\Big), \quad\beta_2 := \frac{4b_2}{p},\quad\beta_3 := \frac{2(p-1)a_5^2}{p}.
\end{align*}
Hence, the inequality $\beta_1 > \beta_2 + \beta_3$ holds when $2b_1>2b_2+(p-1)a_5^2$. In particular, setting 
$\rho = \frac{2b_1-2b_2-(p-1)a_5^2}{2(p-1)(a_5^2+2/p)},$
one has $\alpha_1 > \alpha_2 + \alpha_3$.

Based on \eqref{4-mon} and \eqref{attr+}, we establish, respectively, the uniform-in-time boundedness of the  $p$-th moments and the exponential attractiveness of the segment process associated with \eqref{ISFDE}.
Since the proofs closely follow the arguments in \cite{ WYM17, RS10}, the details are omitted.

\begin{lemma}\label{l4.4}
Let Assumptions \ref{a2.1-f}, \ref{a3.5}, and \ref{dis} hold, and assume that there exists $p \ge 2$ such that $2b_1 > 2b_2 + (p-1)a_5^2$ and $\mu_4, \mu_5 \in \CP_{pr}$. Then for any $t\geq0$,
$$
\EE [ \|x^{\xi}_t\|_r^p ] \le C(1 +  \|\xi\|_r^p e^{-p \alpha t}),
$$
where $\alpha \in (0, \alpha_0)$  is arbitrary with
$
 \alpha_0 := \sup \{ \bar\alpha\in(0,r]: \alpha_1 - 2
 \bar\alpha - \alpha_2 \mu_5^{(p\bar\alpha)} - \alpha_3 \mu_4^{(p\bar\alpha)} > 0 \},
$
and $\alpha_i,i=1,2,3$, are defined in \eqref{4-mon}.
\end{lemma}

\begin{lemma}\label{l4.5}
Let the conditions in Lemma \ref{l4.4} hold. Then for any initial data $\xi, \eta \in \CC_r$ and $t\geq0$, the corresponding solutions $x^{\xi}(t)$ and $x^{\eta}(t)$ satisfy
$$
\EE [ \|x^{\xi}_t-x^{\eta}_t\|_r^p ]  \le  C \|\xi-\eta\|_r^p e^{-p \beta t},
$$ 
where $\beta \in (0,\beta_0)$  is arbitrary with
$
\beta_0 := \sup\{\bar\beta \in(0,r]: \beta_1 - 2\bar\beta - \mu_5^{(p\bar\beta)}\beta_2 - \mu_4^{(p\bar\beta)}\beta_3 > 0 \},$ and 
$\beta_i, i=1,2,3,$ are defined in \eqref{attr+}.
\end{lemma}

Building on Lemmas \ref{l4.4} and \ref{l4.5}, we establish the existence and uniqueness of the IPM for \eqref{ISFDE}.

\begin{theorem}\label{th4.6}
Let the conditions in Lemma \ref{l4.4} hold. Then the Markov semigroup $\{P_t\}_{t \geq 0}$ admits a unique IPM in $\CP_p(\CC_r)$, denoted by $\pi$, satisfying
$$
\WW_p(P^*_t\mu, \pi) \le C\WW_p(\mu,\pi)e^{-\beta t}, \quad \forall~ \m \in \CP_p(\CC_r),~t \ge 0,
$$
where $\beta$ is defined in {\rm Lemma} \ref{l4.5}, $P^*_t$ denotes the transpose operator of $P_t$, and $
P^*_t \mu (A) := \int_{\CC_r} \pi^\xi_t(A) \mu(\mathrm d \xi), \, \forall A \in \mathcal B(\CC_r).
$
\end{theorem}

\begin{proof}
For any $\xi,\eta\in\CC_r$ and $t\geq0$, it follows from Lemma \ref{l4.5} that
\begin{equation*}
\WW_p^p(\pi_{t}^{\xi},\pi_{t}^{\eta})
\le \EE\bigl[\|x_t^\xi-x_t^\eta\|_r^p\bigr]
\le C\|\xi-\eta\|_r^p e^{-p\beta t}.
\end{equation*}
For any $\mu, \nu\in\CP_p(\CC_r)$, the convexity of $\WW_p^p$ yields that
\begin{equation*}\label{Wp-contraction}
\begin{aligned}
\WW^p_p(P_t^*\mu,P_t^*\nu)
& \le \int_{\CC_r\times\CC_r} \WW_p^p(\pi_{t}^{\xi},\pi_{t}^{\eta}) \tilde \pi(\d \xi, \d \eta) \\ 
& \le C e^{-p\beta t} \int_{\CC_r\times\CC_r}\|\xi-\eta\|_r^p  \tilde \pi(\d \xi, \d \eta)
,\quad \forall \tilde \pi \in \Pi(\mu,\nu).
\end{aligned}
\end{equation*}
Taking the infimum over all such couplings, we obtain
$$
\WW_p(P_t^*\mu,P_t^*\nu) \le  C\WW_p(\mu,\nu)e^{-\beta t}.
$$
Choose $T_0>0$ such that $Ce^{-\beta T_0}<1$.
By the preceding estimate and the Banach fixed-point theorem,
$P_{T_0}^*$ admits a unique fixed point $\pi$ in
$\mathcal{P}_p(\mathcal{C}_r)$.
Moreover, by the semigroup property, $P_t^*\pi$ is also
a fixed point of $P_{T_0}^*$ for every $t\geq 0$.
Hence, $P_t^*\pi=\pi$ for all $t\geq 0$.
This proves the existence and uniqueness of the IPM. 
According to  $P_t^*\pi=\pi$, we have
$
\WW_p(P_t^*\mu,\pi)
=\WW_p(P_t^*\mu,P_t^*\pi)\le C\WW_p(\mu,\pi)e^{-\beta t}.
$
The proof is completed.
\end{proof}

\subsection{Convergence of the numerical invariant measure}\label{S5.1}
In this subsection, we approximate the IPM of \eqref{ISFDE} by means of the finite-memory TEM scheme. 
We retain the function $\Gamma$, the projection $\Pi^\Delta$, and the
finite-memory TEM scheme constructed in Section~\ref{S3}.

For any $K\in\mathbb N_+$ and $\Delta\in(0,1]$, we define a space $\mathcal C_{r}^{K,\Delta}(\subset\mathcal C_r)$ containing the numerical functional solution. Specifically,  $\mathcal C_{r}^{K,\Delta}$ denotes the family of all functions $\zeta^{K,\Delta}$ satisfying
	$$
	\zeta^{K,\Delta}(u) = \left\{\begin{aligned}
	& \frac{t_{m+1}-u}{\Delta}\Pi^\Delta(\zeta(t_m)) + \frac{u-t_m}{\Delta}\Pi^\Delta(\zeta(t_{m+1})), \\
    & \qquad t_m \leq u \leq t_{m+1}, -k \leq m \leq -1, \\
	& \Pi^\Delta(\zeta(-K)) , \quad u \leq -K
	\end{aligned}\right.
	$$ 
	for any $\zeta\in\mathcal C_r$, where $\Pi^\Delta$ is the truncated mapping defined in Section~\ref{S3}.
Clearly, $\mathcal C_r^{K,\Delta}$ is a closed subset of the Polish space $(\mathcal C_r,\|\cdot\|_r)$ and is therefore a Polish space under the inherited metric.
Denote by $\{X^{\xi^{K,\Delta},K,\Delta}(t_n)\}_{n \in\mathbb Z}$ and $\{X_{t_n}^{\xi^{K,\Delta},K,\Delta}\}_{n \in\mathbb N}$ the corresponding numerical solution process and segment (functional) process with the initial datum $\xi^{K,\Delta} \in\mathcal{C}_r^{K,\Delta}$, respectively.

Moreover, it follows from \eqref{trun} that for any $\xi, \eta \in \CC_r$,
\begin{equation}\label{6-f-lip}
\begin{aligned}
&|f(X_{t_j}^{\xi^{K,\Delta},K,\Delta})
 -f(X_{t_j}^{\eta^{K,\Delta},K,\Delta})|^2\le L^2\Delta^{-2\theta}
 \|X_{t_j}^{\xi^{K,\Delta},K,\Delta}
 -X_{t_j}^{\eta^{K,\Delta},K,\Delta}\|_r^2,
\end{aligned}
\end{equation}
and
\begin{equation}\label{6-f-lin}
|f(X_{t_j}^{\xi^{K,\Delta},K,\Delta})|^2
\le 2|f(\mathbf0)|^2+2L^2\Delta^{-2\theta}
 \|X_{t_j}^{\xi^{K,\Delta},K,\Delta}\|_r^2.
\end{equation}

\begin{lemma}\label{l-bou}
Let Assumptions \ref{a2.1-f}, \ref{a3.5}, and \ref{dis} hold with
$\mu_4, \mu_5 \in\mathcal P_{2r}$ and $2b_1>2b_2+a_5^2$. Then there exist $\gamma\in(0,2r)$ and  $\Delta_1\in(0,1]$ 
such that, for every $\xi\in\mathcal C_r$,
allowed $\Delta\in(0,\Delta_1]$, and $t\ge0$,
\[
\sup_{K\in\mathbb N_+}\mathbb E\|X_t^{\xi^{K,\Delta},K,\Delta}\|_r^2
\le C\bigl(1+\|\xi\|_r^2e^{-\gamma t}\bigr),
\]
where $C$ is independent of $K$, $\Delta$, $t$, and $\xi$.
\end{lemma}

\begin{proof}
It follows from \eqref{LI} that for any $\xi \in \CC_r$, $K \in \mathbb N_+$, $\Delta \in (0,1]$ and $n \in\mathbb N$,
\begin{equation*}
\begin{aligned}
& \quad\, \| X^{\xi^{K,\Delta},K,\D}_{t_n}\|_r^2  = \sup_{-K \le u \le 0} \left( e^{ru}| X^{\xi^{K,\Delta},K,\D}_{t_n}(u)| \right)^2  \\
& \le \!\!\sup_{-k \le j \le -1} \sup_{t_j < u \le t_{j+1}}\! \left[ e^{2ru} \left( \frac{t_{j+1}\!-\!u}{\Delta} |X^{\xi^{K,\Delta},K,\D}_{t_n}(t_j)|^2 \!+\! \frac{u\!-\!t_j}{\Delta} |X^{\xi^{K,\Delta},K,\D}_{t_n}(t_{j+1})|^2 \right) \right] \\
& \le \sup_{-k \le j \le -1} e^{2r\D} \left[ \left( e^{rt_j}|X^{\xi^{K,\Delta},K,\D}_{t_n}(t_j)|\Big)^2  \vee \Big( e^{rt_{j+1}}|X^{\xi^{K,\Delta},K,\D}_{t_n}(t_{j+1})|  \right)^2 \right],
\end{aligned}
\end{equation*}
which implies that for any $\bar\gamma \in (0, 2r)$,
\begin{equation}\label{5bou-2'}
\begin{aligned}
\EE \big[ \| X^{\xi^{K,\Delta},K,\D}_{t_N}\|_r^2 \big] 
& \le e^{2r\D} \EE \Big[ \sup_{-k \le n \le 0} \left(  e^{rt_n} |X^{\xi^{K,\Delta},K,\D}_{t_N}(t_n)| \right)^2 \Big] \\
& \le e^{2r\D} \EE \Big[  \sup_{-k + N \le n \le N} \left(  e^{rt_{n-N}} |X^{\xi^{K,\Delta},K,\D}(t_n)| \right)^2  \Big] \\
&  \le e^{2r\D} e^{-\bar\gamma t_N} \Big( \|\xi\|_r^2 + \EE \Big[ \sup_{0 \le n \le N} \left( e^{\bar\gamma t_n} |X^{\xi^{K,\Delta},K,\D}(t_n)|^2 \right) \Big] \Big).
\end{aligned}
\end{equation}
By \eqref{TEM} and $1-e^{-x} \leq x$ for each $x \in \RR_+$, we have that for every $j\in\mathbb{N}$,
\begin{equation*}
\begin{aligned}
   & \quad\, e^{\bar\gamma t_{j+1}} \left| X^{\xi^{K,\Delta},K,\D}(t_{j+1}) \right|^2   -  e^{\bar\gamma t_j} \left| X^{\xi^{K,\Delta},K,\D}(t_j) \right|^2  \\
 & \le  \bar\gamma \Delta e^{\bar\gamma t_{j+1}}  \left| X^{\xi^{K,\Delta},K,\D}(t_j) \right|^2 \\
  & \quad + e^{\bar\gamma t_{j+1}}  \left(  2 \big\langle X^{\xi^{K,\Delta},K,\D}(t_j),  f(X^{\xi^{K,\Delta},K,\D}_{t_j}) \big\rangle \D + \big| g(X^{\xi^{K,\Delta},K,\D}_{t_j}) \Delta B_j \big|^2 \right)  \\
      & \quad  + e^{\bar\gamma t_{j+1}} \left( \big| f(X^{\xi^{K,\Delta},K,\D}_{t_j}) \big|^2 \D^2  + R^{\xi^{K,\Delta},K,\Delta}_{1,j} +  R^{\xi^{K,\Delta},K,\Delta}_{2,j} \right),
\end{aligned}
\end{equation*}
where
$$
\begin{aligned}
R^{\xi^{K,\Delta},K,\Delta}_{1,j}& := 2 \big\langle X^{\xi^{K,\Delta},K,\Delta}(t_j), g(X^{\xi^{K,\Delta},K,\Delta}_{t_j}) \Delta B_j \rangle, \  \\
R^{\xi^{K,\Delta},K,\Delta}_{2,j} 
& := 2 \big\langle  f(X^{\xi^{K,\Delta},K,\Delta}_{t_j}) \Delta, g(X^{\xi^{K,\Delta},K,\Delta}_{t_j}) \Delta B_j \rangle.
\end{aligned}
$$
Summing the above inequality over $j=0, \ldots, n-1$ $(n \ge 1)$ yields
\begin{equation}\label{eX}
\begin{aligned}
  & \quad\, e^{\bar\gamma t_n} \left| X^{\xi^{K,\Delta},K,\D}(t_n) \right|^2 \\ 
 & \le  |\xi(0)|^2  +  \bar\gamma \Delta \sum_{j=0}^{n-1} e^{\bar\gamma t_{j+1}}  \left| X^{\xi^{K,\Delta},K,\D}(t_j) \right|^2  \\
 & \quad + \sum_{j=0}^{n-1}e^{\bar\gamma t_{j+1}}  \left(  2 \big\langle X^{\xi^{K,\Delta},K,\D}(t_j),  f(X^{\xi^{K,\Delta},K,\D}_{t_j}) \big\rangle \D \right.  \left. + \big| g(X^{\xi^{K,\Delta},K,\D}_{t_j}) \Delta B_j \big|^2 \right)  \\
& \quad + \sum_{j=0}^{n-1} e^{\bar\gamma t_{j+1}} \left( \big| f(X^{\xi^{K,\Delta},K,\D}_{t_j}) \big|^2 \D^2  + R^{\xi^{K,\Delta},K,\Delta}_{1,j} +  R^{\xi^{K,\Delta},K,\Delta}_{2,j} \right).
\end{aligned}
\end{equation}
By \eqref{4-mon} with $p=2$ and \eqref{6-f-lin}, one has
\begin{align}\nonumber\label{5bou-1'}
   & \quad\, e^{\bar\gamma t_n} \EE \big[ | X^{\xi^{K,\Delta},K,\D}(t_n)|^2 \big]  \\ \nonumber
 & \le |\xi(0)|^2  +  \bar\gamma \Delta \sum_{j=0}^{n-1} e^{\bar\gamma t_{j+1}}  \EE \big[ | X^{\xi^{K,\Delta},K,\D}(t_j) |^2 \big]  +  \sum_{j=0}^{n-1} e^{\bar\gamma t_{j+1}}  \EE \big[ | f(X^{\xi^{K,\Delta},K,\D}_{t_j})|^2 \big] \D^2  \\ \nonumber
   & \quad  +  \Delta \sum_{j=0}^{n-1} e^{\bar\gamma t_{j+1}}  \EE \Big[  2 \big\langle X^{\xi^{K,\Delta},K,\D}(t_j),  f(X^{\xi^{K,\Delta},K,\D}_{t_j}) \big\rangle  + | g(X^{\xi^{K,\Delta},K,\D}_{t_j})|^2 \Big] \\ 
 & \le  \|\xi\|_r^2  +  C e^{\bar\gamma t_{n}} - (\alpha_1 - \bar\gamma) \Delta \sum_{j=0}^{n-1}  e^{\bar\gamma t_{j+1}} \EE \big[ |X^{\xi^{K,\Delta},K,\Delta}(t_j)|^2 \big]  \\ \nonumber
 & \quad + \EE \big[  I^{\xi^{K,\Delta}, K,\Delta}_1 + I^{\xi^{K,\Delta}, K,\Delta}_2 + I^{\xi^{K,\Delta}, K,\Delta}_3 \big],
\end{align}
where 
\begin{align*}
&I^{\xi^{K,\Delta},K,\Delta}_1 := 2 L^2 \Delta^{2-2\theta} \sum_{j=0}^{n-1} e^{\bar\gamma t_{j+1}} \|X^{\xi^{K,\Delta},K,\Delta}_{t_j}\|_r^2,\\
&I^{\xi^{K,\Delta},K,\Delta}_2 :=  \alpha_2 \Delta \sum_{j=0}^{n-1} e^{\bar\gamma t_{j+1}}\int_{-\infty}^0 |X^{\xi^{K,\Delta},K,\Delta}_{t_j}(u)|^2 \mu_5(\d u),\\
&I^{\xi^{K,\Delta},K,\Delta}_3 := \alpha_3 \Delta \sum_{j=0}^{n-1} e^{\bar\gamma t_{j+1}} \int_{-\infty}^0 |X^{\xi^{K,\Delta},K,\Delta}_{t_j}(u)|^2 \mu_4(\d u).
\end{align*}
Similar to the proof of \cite[Theorem 3.1]{LLM}, we have
\begin{equation}\label{5nu1}
\begin{aligned}
 & I^{\xi^{K,\Delta},K,\Delta}_2\le  C \|\xi\|_r^2   +  \alpha_2  e^{\bar\gamma \Delta} \mu_5^{(\bar\gamma)}  \Delta  \sum_{j=0}^{n-1} e^{\bar\gamma t_{j+1}} |X^{\xi^{K,\Delta},K,\Delta}(t_j)|^2,
\end{aligned}
\end{equation}
and
\begin{equation}\label{5nu2}
\begin{aligned}
 & I^{\xi^{K,\Delta},K,\Delta}_3\le  C \|\xi\|_r^2  +  \alpha_3  e^{\bar\gamma \Delta} \mu_4^{(\bar\gamma)}  \Delta  \sum_{j=0}^{n-1} e^{\bar\gamma t_{j+1}} |X^{\xi^{K,\Delta},K,\Delta}(t_j)|^2.
\end{aligned}
\end{equation}
Substituting \eqref{5nu1} and \eqref{5nu2} into \eqref{5bou-1'}, we derive
\begin{equation}\label{eX1}
\begin{aligned}
 & \quad\, e^{\bar\gamma t_n} \EE \big[ |X^{\xi^{K,\Delta},K,\D}(t_n)|^2 \big] \\
 & \le C \left( \|\xi\|_r^2 +   e^{\bar\gamma t_{n}} \right)  - H_{\bar\gamma,\Delta} \Delta \sum_{j=0}^{n-1} e^{\bar\gamma t_{j+1}} \EE \big[ |X^{\xi^{K,\Delta},K,\Delta}(t_j)|^2 \big] + \EE\big[I^{\xi^{K,\Delta},K,\Delta}_1\big],
\end{aligned}
\end{equation}
where
$
H_{\bar\gamma,\D} := \alpha_1 - \bar\gamma - e^{\bar\gamma \Delta} \left(  \mu_5^{(\bar\gamma)} \alpha_2 + \mu_4^{(\bar\gamma)} \alpha_3 \right).
$
Obviously, $H_{0,1}=\alpha_1-\alpha_2-\alpha_3>0$. Thus, we can choose
$\bar\gamma\in(0,2r)$, which is independent of $K$, such that $H_{\bar\gamma,1}>0$.
Then
$
H_{\bar\gamma,\Delta}\ge H_{\bar\gamma,1}>0,
\quad \forall\, 0 < \Delta \leq 1.
$
It follows from \eqref{eX1} that
\begin{equation}\label{eq5.12}
\begin{aligned}
&\Delta\sum_{j=0}^{n-1}e^{\bar\gamma t_{j+1}}
  \EE\big[|X^{\xi^{K,\Delta},K,\Delta}(t_j)|^2\big] \le C\left(\|\xi\|_r^2+e^{\bar\gamma t_n}\right)
 +C\EE\big[I^{\xi^{K,\Delta},K,\Delta}_1\big],
\end{aligned}
\end{equation}
where $C$ depends on $\bar{\gamma}$ but is independent of $K$ and $\Delta$.
By estimates analogous to those used in deriving \eqref{eX1}, together with \eqref{mon-f}, \eqref{6-f-lin}, and \eqref{eX}, we obtain
\begin{equation}\label{eX2}
\begin{aligned}
& \quad\, e^{\bar\gamma t_n} |X^{\xi^{K,\Delta},K,\Delta}(t_n)|^2 \\
 & \le C \left( \|\xi\|_r^2 +   e^{\bar\gamma t_{n}} \right)  - \tilde H_{\bar\gamma, \Delta} \Delta \sum_{j=0}^{n-1} e^{\bar\gamma t_{j+1}} 
      |X^{\xi^{K,\Delta},K,\Delta}(t_j)|^2 + I^{\xi^{K,\Delta},K,\Delta}_1 \\
      & \quad + \sum_{j=0}^{n-1} e^{\bar\gamma t_{j+1}} \left( |g(X^{\xi^{K,\Delta},K,\Delta}_{t_j}) \Delta B_j|^2 + R^{\xi^{K,\Delta},K,\Delta}_{1,j} + R^{\xi^{K,\Delta},K,\Delta}_{2,j}\right),
\end{aligned}
\end{equation}
where 
$
\tilde H_{\bar\gamma, \Delta} = 2b_1 - \rho - \bar\gamma - 2b_2 e^{\bar\gamma \Delta} \mu_5^{(\bar\gamma)}.
$
Clearly, $
\tilde H_{\bar\gamma,\Delta} \ge H_{\bar\gamma,\Delta} > 0$, since $2b_1 > 2b_2 + a_5^2$. This, together with \eqref{eX2} yields
\begin{equation}\label{E-sup-eX1}
\begin{aligned}
 & \quad\, \EE \Big[ \sup_{0 \le n \le N} \left( e^{\bar\gamma t_n} |X^{\xi^{K,\Delta},K,\D}(t_n)|^2\right)  \Big] \\
 & \le C \left( \|\xi\|_r^2 +   e^{\bar\gamma t_{N}} \right)  + 2L^2\Delta^{2-2\theta}\sum_{j=0}^{N-1}e^{\bar\gamma t_{j+1}}\EE\big[\|X^{\xi^{K,\Delta},K,\Delta}_{t_j}\|_r^2\big] \\
 & \quad + \Delta \sum_{n=0}^{N-1} e^{\bar\gamma t_{n+1}} \EE \big[ |g(X^{\xi^{K,\Delta},K,\Delta}_{t_n})|^2 \big] \\
  & \quad  + \EE \bigg[ \sup_{0 \le n \le N}  \sum_{j=0}^{n-1} e^{\bar\gamma t_{j+1}} \left( R^{\xi^{K,\Delta},K,\Delta}_{1,j} + R^{\xi^{K,\Delta},K,\Delta}_{2,j}\right)\bigg].
\end{aligned}
\end{equation}
By the same techniques as in the derivation of \eqref{3-3-16}, we obtain
\begin{align}\label{E-e-R1}\nonumber
& \quad\, \EE \bigg[ \sup_{0 \le n \le N}  \sum_{j=0}^{n-1} e^{\bar\gamma t_{j+1}} \left( R^{\xi^{K,\Delta},K,\Delta}_{1,j} + R^{\xi^{K,\Delta},K,\Delta}_{2,j}\right)\bigg]\\ \nonumber
& \le \frac{1}{2} \EE \Big[ \sup_{0\le n \le N} \left( e^{\bar\gamma t_n} |X^{\xi^{K,\Delta},K,\Delta}(t_n)|^2 \right) \Big] + C \Delta^2  \sum_{j=0}^{N-1} e^{\bar\gamma t_{j+1}} \EE \big[|f(X^{\xi^{K,\Delta},K,\D}_{t_j})|^2\big] \\
& \quad\,  +  C \Delta \sum_{j=0}^{N-1} e^{\bar\gamma t_{j+1}} \EE \big[|g(X^{\xi^{K,\Delta},K,\Delta}_{t_j})|^2\big]
\end{align}
Substituting \eqref{E-e-R1} into \eqref{E-sup-eX1} and making use of \eqref{mon-g}, \eqref{6-f-lin}, \eqref{5nu2} and \eqref{eq5.12}, one derives that 
\begin{align*}
& \quad\, \EE\Big[\sup_{0\le n\le N}e^{\bar\gamma t_n}|X^{\xi^{K,\Delta},K,\Delta}(t_n)|^2\Big]\\
&\le C\left(\|\xi\|_r^2+e^{\bar\gamma t_N}\right)
 +C\Delta^2\sum_{j=0}^{N-1}e^{\bar\gamma t_{j+1}}
   \EE\big[|f(X^{\xi^{K,\Delta},K,\Delta}_{t_j})|^2\big]\\
&\quad+C\Delta\sum_{j=0}^{N-1}e^{\bar\gamma t_{j+1}}
   \EE\big[|g(X^{\xi^{K,\Delta},K,\Delta}_{t_j})|^2\big]+C\Delta^{2-2\theta}\sum_{j=0}^{N-1}e^{\bar\gamma t_{j+1}}
   \EE\big[\|X^{\xi^{K,\Delta},K,\Delta}_{t_j}\|_r^2\big]\\
& \le C\left(\|\xi\|_r^2+e^{\bar\gamma t_N}\right) + C\Delta^{2-2\theta}\sum_{j=0}^{N-1}e^{\bar\gamma t_{j+1}}
\EE\big[\|X^{\xi^{K,\Delta},K,\Delta}_{t_j}\|_r^2\big].
\end{align*}
By \eqref{5bou-2'}, for every $j\ge0$, one has
\begin{align*}
e^{\bar\gamma t_{j+1}}\EE\big[\|X^{\xi^{K,\Delta},K,\Delta}_{t_j}\|_r^2\big]\le e^{(2r+\bar\gamma)\Delta}\left(
 \|\xi\|_r^2+
 \EE\Big[\sup_{0\le n\le j}e^{\bar\gamma t_n}|X^{\xi^{K,\Delta},K,\Delta}(t_n)|^2\Big]\right).
\end{align*}
Therefore, for every $N \geq 1$,
\begin{align*}
&\quad\,\sup_{0\le n\le N}\left\{e^{-\bar\gamma t_n/2}\EE\Big[\sup_{0\le j \le n}e^{\bar\gamma t_j}|X^{\xi^{K,\Delta},K,\Delta}(t_j)|^2\Big]\right\}\\
&\le C\left(\|\xi\|_r^2+e^{\bar\gamma t_N/2}\right) +C 
 \Delta^{1-2\theta}
 \sup_{0\le n\le N}\left\{
 e^{-\bar\gamma t_n/2}
 \EE\Big[\sup_{0\le j\le n}e^{\bar\gamma t_j}|X^{\xi^{K,\Delta},K,\Delta}(t_j)|^2\Big]\right\}.
\end{align*}
Here, we used the facts that
$$
t_ne^{-\bar\gamma t_n/2}\le\frac{2}{e\bar\gamma} \quad 
\Delta\sum_{j=0}^{n-1}e^{-\bar\gamma(t_n-t_j)/2}
=
\frac{\Delta(1-e^{-\bar\gamma t_n/2})}
{e^{\bar\gamma\Delta/2}-1}
\le\frac{2}{\bar\gamma}.
$$
Choose $\Delta_1$ small enough such that $C \Delta_1^{1-2\theta} < 1/2$.
Then, for any $0<\Delta\le\Delta_1$, we have
$
\EE\Big[\sup_{0\le n\le N}e^{\bar\gamma t_n}|X^{\xi^{K,\Delta},K,\Delta}(t_n)|^2\Big]
\le C\left(\|\xi\|_r^2e^{\bar\gamma t_N/2}+e^{\bar\gamma t_N}\right).
$
Combining this with \eqref{5bou-2'} and taking $\gamma=\bar\gamma/2$, we obtain
$
\sup_{K\in\mathbb N_+}\EE\big[\|X^{\xi^{K,\Delta},K,\Delta}_{t_N}\|_r^2\big]
\le C\left(1+\|\xi\|_r^2e^{-\gamma t_N}\right).
$
Finally, \eqref{step} and $e^{-\gamma t_N}\le e^{\gamma-\gamma t}$
for $t_N\le t<t_{N+1}$ extend the above estimate to all $t\ge 0$.
This completes the proof.
\end{proof}

\begin{lemma}\label{l-att}
Under the conditions of Lemma~\ref{l-bou}, there exist $\lambda\in(0,2r)$ and
$\Delta_2\in(0,1]$ such that, for any $\xi,\eta\in\mathcal C_r$, $K\in\mathbb N_+$,
$\Delta\in(0,\Delta_2]$, and $n\in\mathbb N$,
\begin{equation}\label{5att-2}
\mathbb E\|X_{t_n}^{\xi^{K,\Delta},K,\Delta}
-X_{t_n}^{\eta^{K,\Delta},K,\Delta}\|_r^2
\le C\|\xi-\eta\|_r^2e^{-\lambda t_n},
\end{equation}
where $C$ is independent of
$K$, $\Delta$, $n$, $\xi$, and $\eta$.
\end{lemma}

\begin{proof}
We apply the argument used in the proof of Lemma \ref{l-bou} to the difference between the two numerical solutions. Using \eqref{6-f-lip} and Assumptions \ref{a3.5} and \ref{dis} then yields the desired estimate. The details are omitted.
\end{proof}

We introduce some notation that will be used to establish the existence and uniqueness of the numerical IPM and its convergence to the exact IPM. 
Denote by $\mathcal B(\CC_r^{K,\Delta})$ the Borel
$\sigma$-algebra on $\CC_r^{K,\Delta}$ and by
$\mathcal B_b(\CC_r^{K,\Delta})$ the space of all
bounded Borel-measurable functions on $\CC_r^{K,\Delta}$.
Let $\mathcal P_2(\mathcal C_r^{K,\Delta})$ denote the set of probability measures $\mu$ on $\mathcal C_r^{K,\Delta}$ satisfying
$\mu(\|\cdot\|_r^2)<\infty.$
For $\xi\in\mathcal C_r$ and $n\in\mathbb N$, denote by $\pi_{t_n}^{\xi^{K,\Delta},K,\Delta}$ the probability measure induced by $X_{t_n}^{\xi^{K,\Delta},K,\Delta}$, that is,
\[
\pi_{t_n}^{\xi^{K,\Delta},K,\Delta}(A)
:=
\mathbb P\left\{
\omega\in\Omega:
X_{t_n}^{\xi^{K,\Delta},K,\Delta}\in A
\right\},
\qquad
A\in\mathcal B(\mathcal C_r^{K,\Delta}).
\]
Define the linear operator $
P_{t_n}^{K,\Delta}:
\mathcal B_b(\mathcal C_r^{K,\Delta})
\longrightarrow
\mathcal B_b(\mathcal C_r^{K,\Delta})
$
associated with the numerical segment process by
\[
\begin{aligned}
& (P_{t_n}^{K,\Delta}\Psi)(\xi^{K,\Delta})
:=
\mathbb E \left[
\Psi \left(X_{t_n}^{\xi^{K,\Delta},K,\Delta}\right)
\right]  =
\int_{\mathcal C_r^{K,\Delta}}
\Psi(\phi)\,
\pi_{t_n}^{\xi^{K,\Delta},K,\Delta}(\mathrm d\phi)
\end{aligned}
\]
for any $n\in\mathbb N$ and $
\Psi\in\mathcal B_b(\mathcal C_r^{K,\Delta})$.

By standard arguments (see, e.g., \cite{SWMW24}), the sequence $\{X^{\xi^{K,\Delta},K,\Delta}_{t_n}\}_{n \in\mathbb N}$ is a time-homogeneous Markov chain.
Namely, under Assumptions \ref{a2.1-f}\text{-}\ref{a2.2}, for any $K \in \mathbb N_+$, $\xi \in \CC_r$, and $\Delta \in (0,1]$, $\{X^{\xi^{K,\Delta},K,\Delta}_{t_n}\}_{n \in\mathbb N}$ satisfies
\[
\mathbb E\left[
\Psi\left(X_{t_{n+1}}^{\xi^{K,\Delta},K,\Delta}\right) \Big| \mathcal F_{t_n}\right]
=
\mathbb E\left[
\Psi\left(X_{t_{1}}^{\eta,K,\Delta}\right) 
\right] \Big|_{\eta = X_{t_n}^{\xi^{K,\Delta},K,\Delta}} 
\quad\text{a.s.}
\]
for any $\Psi \in \mathcal B_b(\CC^{K,\Delta}_r)$.
Consequently, the linear operator $
P_{t_n}^{K,\Delta}$ is a discrete-time Markov semigroup.
Moreover, Lemma~\ref{l-att} implies that this semigroup has the Feller property.

\begin{theorem}\label{th5.4}
Let the conditions in Lemma \ref{l-bou} hold. Then for any $K \in \NN_+$, $\check\Delta \in (0, \Delta_1\wedge\Delta_2)$ and $\Delta \in (0, \check\Delta]$, $\{P^{K,\Delta}_{t_n}\}_{n \ge 0}$ admits a unique IPM $\pi^{K,\Delta} \in \mathcal{P}_2(\mathcal{C}^{K,\Delta}_r)$. Moreover, 
  $$
 \mathbb{W}_2((P^{K,\Delta}_{t_n})^{*}\mu, \pi^{K,\Delta}) \le C\left( 1+\left(\mu(\|\cdot\|_r^2)\right)^{1/2} \right)e^{-\frac{\lambda}{2}t_n}, \quad \forall\,\mu \in \CP_2(\CC^{K,\Delta}_r),
  $$
where $\Delta_1$ is defined in Lemma \ref{l-bou}, while $\Delta_2$ and $\lambda$ are defined in Lemma \ref{l-att}, $(P^{K,\Delta}_{t_n})^{*}$ denotes the transpose operator of $P^{K,\Delta}_{t_n}$ with 
$$
  (P^{K,\Delta}_{t_n})^{*} \mu (A) := \int_{\CC^{K,\Delta}_r} \pi^{\xi,K,\Delta}_{t_n}(A) \mu (\mathrm d \xi), \quad \forall A \in \mathcal B(\CC^{K,\Delta}_r).
$$
\end{theorem}

\begin{proof}
Fix arbitrary $K \in \NN_+$ and $\check\Delta \in (0, \Delta_1\wedge\Delta_2)$. For any $\Delta \in (0, \check\Delta]$, making use of Lemmas \ref{l-bou}, \ref{l-att}, and techniques used in the proof of Theorem \ref{th4.6}, we conclude that $\{P^{K,\Delta}_{t_n}\}_{n \geq 0}$ admits a unique invariant measure $\pi^{K,\Delta} \in \mathcal{P}_2(\mathcal{C}^{K,\Delta}_r)$ satisfying
\begin{equation}\label{4-39}
  \mathbb{W}_2((P^{K,\Delta}_{t_n})^{*}\mu, \pi^{K,\Delta}) \le C \left(1+ \left(\mu(\|\cdot\|_r^2)\right)^{1/2} + \left(\pi^{K,\Delta}(\|\cdot\|_r^2)\right)^{1/2} \right) e^{-\frac{\lambda}{2}t_n},
\end{equation}
for any $\mu \in \mathcal{P}_2(\CC^{K,\Delta}_r)$ and $n \geq 0$. 
Here $C$ is a positive constant independent of $K$, $\Delta$ and $n$. 
In addition, for any $M > 0$, it follows from Lemma \ref{l-bou} that
$$
\pi^{K,\Delta}(\|\cdot\|_r^2 \wedge M) 
= \lim_{n\rightarrow \infty} \EE\big[ \|X^{\textbf{0},K,\Delta}_{t_n}\|_r^2  \wedge M\big]\leq C,$$
where $C$ is a positive constant independent of $K$, $\Delta$, $n$, and $M$.
Letting $M \to +\infty$ and applying the monotone convergence theorem yield $\pi^{K,\Delta}(\|\cdot\|^2_r) \le C$. Since $C$ is independent of $K$ and $\Delta$, it follows that
$
\sup_{K\in\mathbb N_+}  \pi^{K,\Delta}(\|\cdot\|^2_r) \le C.
$
Combining this with \eqref{4-39} implies the desired assertion.
This completes the proof.
\end{proof}

\begin{theorem}\label{IPM-conv}
Let Assumptions \ref{a2.1-f}, \ref{a3.5}, and \ref{dis} hold with
$2b_1>2b_2+a_5^2$. Then
\[
\lim_{K\to\infty,\,\Delta\to0}\WW_2(\pi^{K,\Delta},\pi)=0.
\]
\end{theorem}

\begin{proof}
For any $K\in\mathbb N_+$, $\check\Delta \in (0,\Delta_1\wedge\Delta_2)$ and $\Delta \in (0,\check\Delta]$,
it follows from Theorems \ref{th4.6} and \ref{th5.4} that for any $n \ge 0$,
\begin{equation}\label{W-k-kD}
\begin{aligned}
 \WW_2(\pi^{K,\Delta}, \pi) & \le   \WW_2(\pi^{K,\Delta},\pi^{\mathbf 0,K,\Delta}_{t_n}) +\WW_2(\pi^{\mathbf 0,K,\Delta}_{t_n}, \pi^{\mathbf 0}_{t_n} ) + \WW_2(\pi^{\mathbf 0}_{t_n}, \pi)   \\
     &  \le   C \left(1 + \pi(\|\cdot\|_r^2) \right) e^{-\left( \frac{\lambda}{2} \wedge \beta \right) t_n} + \left( \EE\big[\|X^{\mathbf{0},K,\Delta}_{t_n}-x^{\mathbf{0}}_{t_n}\|_r^2\big]\right)^{\frac{1}{2}}.
\end{aligned}
\end{equation}
For any $\delta>0$, choose $T_\delta$ sufficiently large that
$
C \left(1 + \pi(\|\cdot\|_r^2) \right) e^{-\left( \frac{\lambda}{2} \wedge \beta \right) T_\delta} \le \frac{\delta}{2}.
$
For any $\Delta \in (0,\check\Delta]$, set $n_\Delta = \lfloor T_\delta/\Delta \rfloor +1$. Clearly, $T_\delta \leq t_{n_\Delta} < T_\delta + 1$. Hence, letting $n = n_\Delta$ in \eqref{W-k-kD} implies
$$
\begin{aligned}
\WW_2(\pi^{K,\Delta}, \pi) 
&  \le   C \left(1 + \pi(\|\cdot\|_r^2) \right) e^{-\left( \frac{\lambda}{2} \wedge \beta \right) T_\delta} + \left( \EE\big[\|X^{\mathbf{0},K,\Delta}_{t_{n_\Delta}}-x^{\mathbf{0}}_{t_{n_\Delta}}\|_r^2\big]\right)^{\frac{1}{2}} \\
&  \le   \frac{\delta}{2} + \left( \sup_{0 \leq t \leq T_\delta+1} \EE\big[\|X^{\mathbf{0},K,\Delta}_t-x^{\mathbf{0}}_t\|_r^2\big]\right)^{\frac{1}{2}}. \\
\end{aligned}
$$
By Theorem \ref{th3.4}, there exist $K_\delta > 1$ and $\Delta_\delta > 0$ such that for any $K \geq K_\delta$ and any $\Delta \in (0,\Delta_\delta\wedge\check\Delta)$,
$
\left( \sup_{0 \leq t \leq T_\delta+1} \EE \big[\|X^{\mathbf{0},K,\Delta}_{t}-x^{\mathbf{0}}_{t}\|_r^2 \big]\right)^{\frac{1}{2}} \le \frac{\delta}{2}.
$
This completes the proof.
\end{proof}

\subsection{Convergence rate}
This subsection establishes the convergence rate of the numerical IPM under slightly stronger conditions. 
Under Assumption \ref{a3.3}, fix \(\Gamma,L,\theta\) as specified in Remark \ref{rmk3.12} with \(q=2\). Set 
$
K_{\D} = \left\lceil \frac{-\ln(\D/M)}{\alpha_2} \right\rceil \vee 1,
$
as defined in \eqref{k_De} with \(q=2\), where $\alpha_2$ is defined in Theorem \ref{th3.20}. Throughout this subsection, we keep the above parameters fixed.
Denote by $X^\Delta_{t_n} := X^{K_\Delta,\Delta}_{t_n}$ the corresponding finite-memory TEM numerical segment process. Let $\{P_{t_n}^{\Delta}\}_{n \in\mathbb N}$ and $\pi^\Delta$ denote the discrete-time Markov semigroup and IPM associated with the numerical segment process $\{X^\Delta_{t_n}\}_{n \in\mathbb N}$, respectively.

\begin{theorem}\label{IPM-rate}
Let Assumptions \ref{a2.1-f}, \ref{a3.5}, \ref{a3.3}, and \ref{dis} hold with 
$2b_1 > 2b_2 + (v \vee 1) a_5^2$, and $\mu_5,~\mu_4 \in \CP_{\bar c r}$ for some $\bar c > 2$. Then for any $\hat\Delta \in (0,\Delta_1\wedge\Delta_2)$ and $\varepsilon \in (0,1/2)$, there exists $\rho_\varepsilon>0$ such that for any $\Delta \in (0,\hat\Delta]$,
$$
 \WW_2(\pi^{\Delta}, \pi) \le C \Delta^{\rho_\varepsilon}. 
$$
\end{theorem}

\begin{proof}
Fix $\hat\Delta \in (0,\Delta_1\wedge\Delta_2)$ arbitrarily. Using the techniques used in the proof of \eqref{W-k-kD} and applying Corollary \ref{rmk3.19} with $q=2$, we obtain that there exist positive constants $C_1$, $C_2$ and $\zeta$, such that for any $\varepsilon \in (0,1/2)$, $\Delta \in (0,\hat\Delta]$ and $T>0$,
\begin{equation*}
\begin{aligned}
    \WW_2(\pi^{\Delta}, \pi) & \le   C_1  e^{-\bigl(\frac{\lambda}{2} \wedge \beta\bigr) T} + \Big( \sup_{0 \leq t \leq T+1} \EE\big[\|X^{\mathbf{0},\Delta}_{t}-x^{\mathbf{0}}_{t}\|_r^2\big]\Big)^{\frac{1}{2}} \\
   & \le  C_1 e^{-\bigl(\frac{\lambda}{2} \wedge \beta\bigr) T} + C_2 e^{\zeta T}  \Delta^{\frac{1}{2}-\varepsilon}.
\end{aligned}
\end{equation*}
Setting
$
\rho_{\varepsilon} = \frac{\bigl(\frac{\lambda}{2} \wedge \beta\bigr)(1-2\varepsilon)}{2\bigl(\bigl(\frac{\lambda}{2} \wedge \beta\bigr) + \zeta \bigr)}, \quad
T_{\varepsilon,\Delta} = \frac{-\rho_{\varepsilon} \ln \Delta}{\frac{\lambda}{2} \wedge \beta} \vee 1,
$
we have that
$$
\begin{aligned}
    \WW_2(\pi^{\Delta}, \pi) \le C_1 e^{-\left(\frac{\lambda}{2}\wedge\beta\right)T_{\varepsilon,\Delta}} + C_2 e^{\zeta T_{\varepsilon,\Delta}} \Delta^{\frac{1}{2}-\varepsilon}  \le  C \Delta^{\rho_\varepsilon}.
\end{aligned}
$$
This completes the proof.
\end{proof}

\section{Numerical experiments}\label{S5}
This section presents numerical experiments to illustrate the finite-time strong convergence, ergodicity of the finite-memory TEM scheme \eqref{TEM}.

\begin{expl}
Consider the following SFDE with infinite delay:
\begin{equation}\label{eg7.1}
   \d x(t) = \left( d_1 -  d_2 x(t) - d_3 x^3(t) + d_4 M(x_t)  \right) \d t + d_5 M(x_t) \d B(t)
\end{equation}
with three initial data $\xi_1(u) = e^{0.2u}$, $\xi_2(u) = -e^{0.2u}$, and $\xi_3(u)=u$. Here, $M(\phi)=\int_{-\infty}^0 \phi(u) \tilde\mu(\d u)$, $\tilde\mu(\d u) = 3e^{3u}\d u$, $d_1 \in \mathbb{R}$, and $d_2,\ldots,d_5$ are positive constants satisfying $d_2 > d_4 + d_5^2$. It is straightforward to verify that $\xi_1, \xi_2, \xi_3 \in \CC_{0.3}$  and that Assumptions \ref{a2.1-f}, \ref{a3.5}, and \ref{a5.8} hold.
Moreover, Assumptions \ref{a3.3} and \ref{dis} hold with
$$a_8 = d_2 + \frac{3d_3}{2} + \frac{10d_4}{9}, \ \   v =2, \ \  b_1 = d_2 - \frac{d_4}{2}, \ \ 
 b_2 = \frac{d_4}{2}, \ \  a_5 = d_5, \ \ \mu_5 = \mu_4 = \tilde\mu.
$$
The condition $d_2 > d_4 + d_5^2$ then implies $b_1 > b_2 + a_5^2$. 
In the numerical experiments, we set $d_1 = d_5 = 1$, $d_2 = 8$, $d_3 = 2$, and $d_4 = 6$.
Following Remark \ref{rmk3.12}, we take
$
\theta = \frac{2}{5}, \ \   \Gamma(R) = \frac{106}{3}(R^2+1/2), \  L = 18.
$

In this example, choose $c_1=9$, $\varepsilon = 0.1$, $\alpha_2 = 0.5$. Let $M=1$ and $K_\Delta=\lceil-2\ln\Delta\rceil\vee1$. 
We denote by $\{X^{\Delta}_{t_n}\}_{n\in\mathbb N_+}$ the corresponding numerical segment process.

\paragraph{\textbf{Experiment $1$: Strong convergence}}
We first illustrate the time-discretization error established in Corollary \ref{rmk3.19}. Since \eqref{eg7.1} has no explicit solution, using the finite-memory TEM numerical segment, we regard the finer approximation with $\Delta=2^{-15}$ as the reference segment $x_t$, and compare it with the numerical segment $X^{\Delta}_t$ with $\Delta \in \{2^{-8}, 2^{-9}, 2^{-10}, 2^{-11}, 2^{-12}\}$. 
To compute the error, we run $S = 2000$ independent trajectories where $x^j_t$ and $X^{\Delta,j}_t$ denote the $j$th reference and numerical segments. Thus
$$
 \big( \EE \big[ \|X^{\Delta}_{t} - x_{t}\|_r^2 \big] \big)^{1/2} \approx \bigg( \bigg[\frac{1}{S}\sum_{j=1}^S \|X^{\Delta,j}_t - x^j_t\|^2_r \bigg] \bigg)^{\frac{1}{2}}.
$$
Figure \ref{5-tu1} shows the root mean square approximation error as a function of $\Delta$ at $t = 10$. The observed convergence rate is close to $1/2$.
\begin{figure}[htp]
	\centering
	\includegraphics[width=10cm]{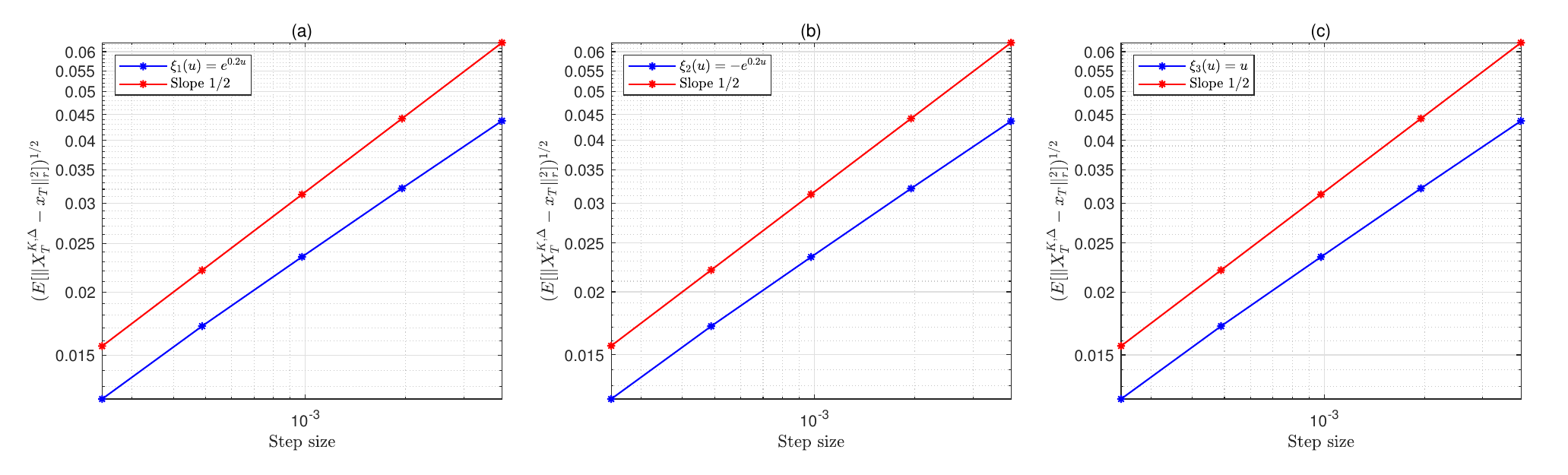}
	\caption{(a) $\xi_1(u) = e^{0.2u}$. (b) $\xi_2(u) = - e^{0.2u}$. (c) $\xi_3(u) = u$.}\label{5-tu1}
\end{figure}

\paragraph{\textbf{Experiment 2: Numerical ergodicity.}}
We next illustrate numerical ergodicity established in Theorem \ref{th5.4}.
Set $\Delta=2^{-6}$, $S=2000$. For test functionals
$
F_1(\varphi)=\cos(\|\varphi\|_r),\ 
F_2(\varphi)=\|\varphi\|_r\wedge 2,
$
using independent numerical trajectories, for each initial segment $\xi_i$, we have
$$
\EE[F_m(X^\Delta_t)] \approx \frac1S\sum_{j=1}^{S}F_m(X_{t}^{\Delta,j}).
$$
For each of the numerical experiments performed, the blue, red, and green curves correspond to the sample means obtained from the initial segments $\xi_1$, $\xi_2$, and $\xi_3$, respectively. Figure \ref{5-tu2} displays these sample means over the time interval $[0,15]$.
The curves corresponding to different initial data approach the same
limiting level as time increases. This behavior numerically illustrates
ergodicity of the TEM scheme.
\begin{figure}[htp]
	\centering
	\includegraphics[width=12cm]{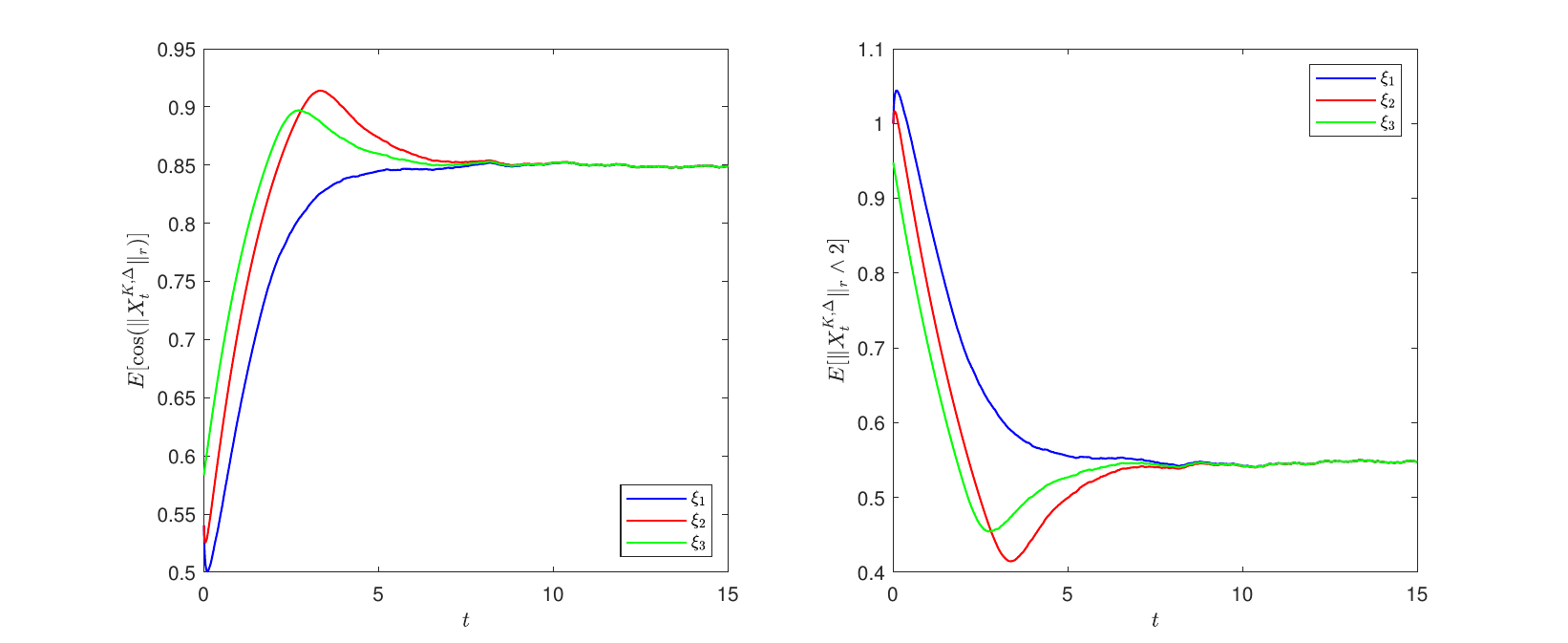}
	\caption{(a) $\EE[\cos(\|\cdot\|_r)]$. (b) $\EE[\|\cdot\|_r \wedge 2]$.}\label{5-tu2}
\end{figure}
\end{expl}

\begin{appendices}
\section{Proofs of Theorem \ref{th2.3}, Lemmas \ref{l3.6} and \ref{lemma3.6}}\label{appA}

\begin{proof}[\textbf{Proof of Theorem \ref{th2.3}.}]
The existence and uniqueness of the global solution follow from \cite{LLMS23}.
For any $t \in [0, T]$, one has
$$
\begin{aligned}
\|x_t\|_r^2 & = \sup_{u \leq 0} \left( e^{ru}|x(t+u)|\right)^2 
=   \sup_{u \leq t} \left( e^{r(u-t)}|x(u)|\right)^2 \\
& \le \sup_{u \leq 0} \left( e^{r(u-t)}|\xi(u)|\right)^2 + \sup_{0 \leq u \leq t} \left( e^{r(u-t)}|x(u)|\right)^2 \\
& \le e^{-2rt} \|\xi\|_r^2 + \sup_{0 \leq u \leq t} |x(u)|^2,
\end{aligned}
$$
which yields that for  $p\geq2$,
\begin{equation}\label{app1}
\EE \left[\sup_{0 \le t \le T} \|x_t\|_r^p \right]
\le  C\|\xi\|_r^p  + C \EE \left[\sup_{0 \le t \le T} |x(t)|^p \right].
\end{equation}
By It$\hat{\rm o}$'s formula and Assumption \ref{a2.2}, we derive that
\begin{equation*}
\begin{aligned}
& \quad\,\, |x(t)|^2 
  \le\! |\xi(0)|^2 \!+ \!\int_0^t  \left( 2\langle x(s), f(x_s) \rangle \!+\! |g(x_s)|^2 \right)  \d s \! +\! 2 \int_0^t \langle x(s), g(x_s)\d B(s) \rangle \\
& \le \|\xi\|_r^2 +  \int_0^t  \left( 3a_1 (1 + \|x_s\|_r^2) - 2a_2|x(s)|^{\bar p} + 2a_2 \int_{-\infty}^0 |x_s(u)|^{\bar p} \mu_1(\d u) \right)  \d s \\
& \quad + 2 \int_0^t \langle x(s), g(x_s)\d B(s) \rangle \\
& \le C \left( \|\xi\|_r^2 + \|\xi\|_r^{\bar p} \right) + CT + C \int_0^t \|x_s\|_r^2 \d s  + 2 \int_0^t \langle x(s), g(x_s)\d B(s) \rangle,
\end{aligned}
\end{equation*}
which implies
\begin{equation}\label{E-sup-x'}
\begin{aligned}
\EE \left[ \sup_{0 \le t \le T} |x(t)|^p \right]  
& \le C \left( \|\xi\|_r^p + \|\xi\|_r^{\frac{p \bar p}{2}} \right) + C T^{\frac{p}{2}} + C T^{\frac{p-2}{2}} \int_0^T  \EE  \|x_s\|_r^p \d s  \\
& \quad\,\, +  8^{\frac{p}{2}}  \EE \left[ \sup_{0 \le t \le T} \left|\int_0^t  \langle x(s), g(x_s)\d B(s) \rangle \right|^{\frac{p}{2}}  \right].
\end{aligned}
\end{equation}
Making use of the Burkholder-Davis-Gundy inequality, the Young inequality, and Assumption \ref{a2.2}, one obtains
\begin{align*}
 & \quad\,\, 8^{\frac{p}{2}} \EE \left[\sup_{0 \le t \le T} \left|\int_0^t \langle x(s), g(x_s)\d B(s) \rangle\right|^{\frac{p}{2}} \right]  
\\
 & \le  \frac{1}{2} \EE \left[\sup_{0 \le t \le T} |x(t)|^p \right] + CT^{\frac{p}{2}} + CT^{\frac{p-2}{2}} \EE \int_0^T  \|x_t\|_r^{p} \d t.
\end{align*}
Inserting the above inequality into \eqref{E-sup-x'}, we have
$$
\EE \left[\sup_{0 \le t \le T} |x(t)|^p \right] 
\le C\left( \|\xi\|_r^p  + \|\xi\|_r^{\frac{p\bar p}{2}} \right) +   C T^{\frac{p}{2}}  +  C T^{\frac{p-2}{2}} \int_0^T \EE \|x_t\|_r^p \d t.
$$
This, together with \eqref{app1}, yields
$$
\EE \left[\sup_{0 \le t \le T} \|x_t\|_r^p \right]  \le  C\left( \|\xi\|_r^p  + \|\xi\|_r^{\frac{p\bar p}{2}} \right) +   C T^{\frac{p}{2}}  +  C T^{\frac{p-2}{2}} \int_0^T \EE \left[\sup_{0 \leq s \leq t} \|x_s\|_r^p \right] \d t.
$$
An application of Gronwall Lemma gives
$
\EE \left[\sup_{0 \le t \le T} \|x_t\|_r^p \right]  \le C (1+\|\xi\|_r^p + \|\xi\|_r^{\frac{p \bar p}{2}}) e^{CT^{\frac{p}{2}}}.
$
For $p\in(0,2)$, the desired estimate follows from H\"older’s inequality and the result established above for $p\geq2$. This completes the proof for all $p>0$.
\end{proof}

\begin{proof}[\textbf{Proof of Lemma \ref{l3.6}}]
We first prove the result for $q>2$. For the case of $0<q\leq2$, it follows
	from H\"older's inequality. Fix $K\in\mathbb{N}_+$, $T>0$, and
	$\Delta\in(0,1]$. For any $t \in [0,T]$ and $u \in [-K,0)$, there exist integers $j \ge 0$ and $m \in \{-k, \cdots, -1\}$ such that $t \in [t_j, t_{j+1})$ and $u \in [t_m, t_{m+1})$. Then, $t+u \in [t_{j+m}, t_{j+m+2})$. Hence, $[-K, 0) = \cup_{-k \le m \le -1} \left( \mathcal T^t_{1,m} \cup \mathcal T^t_{2,m} \right) $, where
	\begin{align} \nonumber
	& \mathcal T^t_{1,m} :=  \left\{u \mid  t+u \in [t_{j+m}, t_{j+m+1})  \right\} \cap [t_m,t_{m+1}), \\ \nonumber
	& \mathcal T^t_{2,m} := \left\{u \mid   t+u \in [t_{j+m+1}, t_{j+m+2}) \right\} \cap [t_m,t_{m+1}).
	\end{align}
	This yields 
	$
	\EE \Big[ \sup_{-K \le u \le 0} \Big(e^{ru} |X^{K,\Delta}_t(u) - Z^{K,\Delta}_t(u) | \Big)^q  \Big]
	\le \sum_{i=1}^{6} E_i, 
	$
	where
	\begin{align*}
	E_1 &:= \EE \Big[  \sup_{u \in \cup_{-k \le m \le  -j-1} \mathcal T^t_{1,m} } \Big(  e^{ru} | X^{K,\Delta}_t(u) -  Z^{K,\Delta}_t(u) | \Big)^q \Big], \\
	E_2 & :=\EE \Big[ \sup_{u \in \cup_{-j \le m \le -1}  \mathcal T^t_{1,m}} \Big( e^{ru} |  X^{K,\Delta}_t(u) -  Z^{K,\Delta}_t(u) | \Big)^q \Big],   \\
	E_3 & := \EE \Big[ \sup_{u \in \cup_{-k \le m \le-j-2} \mathcal T^t_{2,m}} \Big( e^{ru} | X^{K,\Delta}_t(u) - Z^{K,\Delta}_t(u) | \Big)^q \Big], \\
	E_4 & := \EE \Big[ \sup_{ u \in \mathcal T^t_{2, -j-1}} \Big( e^{ru} |  X^{K,\Delta}_t(u) -  Z^{K,\Delta}_t(u) |  \Big)^q \Big],  \\
	E_5 & := \EE \Big[ \sup_{u \in \cup_{-j \le m \le -1}\mathcal T^t_{2,m}} \Big( e^{ru} | X^{K,\Delta}_t(u) -  Z^{K,\Delta}_t(u) |  \Big)^q \Big],  \\
	E_6 & := \EE \big[ | X^{K,\Delta}(t) - Z^{K,\Delta}(t) |^q \big].  
	\end{align*}
	Next, we estimate $E_1,\ldots,E_6$.
	
	\textbf{Estimates for $E_1$ and $E_3$.} For any integer $m$ with $-k\leq m\leq-j-1$ and any $u\in\mathcal T^t_{1,m}$, Jensen's inequality, together with \eqref{LI} and \eqref{ap}, yields
	\begin{align} \nonumber
	& \quad |Z^{K,\Delta}_t(u)-X^{K,\Delta}_t(u)|^q  \\ \nonumber
	& = \left|\frac{t_{m+1}-u}{\Delta} \Pi^{\Delta}\left( \xi(t_{j+m}) \right) +  \frac{u-t_m}{\Delta} \Pi^{\Delta} \left( \xi(t_{j+m+1}) \right) - \xi(t+u) \right|^q \\ \nonumber
	& \le \frac{t_{m+1}-u}{\Delta}  \left| \Pi^{\Delta}(\xi(t_{j+m})) - \xi(t+u)\right|^q + \frac{u-t_m}{\Delta}   \left|\Pi^{\Delta}( \xi(t_{j+m+1}))-\xi(t+u) \right|^q \\ \nonumber
	& \le C \Big( \left| \xi(t_{j+m})-\xi(t+u) \right|^q +  \left| \Pi^{\Delta}(\xi(t_{j+m}))-\xi(t_{j+m}) \right|^q \\ \nonumber
	& \quad +   \left| \xi(t_{j+m+1})- \xi(t+u) \right|^q +   \left| \Pi^{\Delta}( \xi(t_{j+m+1}))-\xi(t_{j+m+1}) \right|^q \Big) \\ \nonumber
	& \le C \left( \mathcal H^q(\Delta) + \left| \Pi^{\Delta}(\xi(t_{j+m}))-\xi(t_{j+m}) \right|^q + \left|\Pi^{\Delta}( \xi(t_{j+m+1}))-\xi(t_{j+m+1}) \right|^q \right).
	\end{align}
	By Assumption \ref{xi1}, the uniform continuity of $\xi$ implies that 
	there exists a constant $C_\xi$ such that $|\xi(u)| \leq C_\xi (1 + |u|)$ for any $u \le 0$. Therefore, for any $\gamma>0$,
	\begin{equation}\label{xi}
	\begin{aligned}
	|\Pi^{\Delta}(\xi(u))-\xi(u) |^q 
	& = |\Pi^{\Delta}(\xi(u))-\xi(u)|^q \II_{\{|\xi(u)| \ge \Gamma^{-1}(L\Delta^{-\theta})\}} \\
	& \le \frac{2^q|\xi(u)|^{q+\gamma}}{\left(\Gamma^{-1}(L\Delta^{-\theta})\right)^\gamma}  \le \frac{C_\xi  \left( 1 + |u|^{q+\gamma} \right)}{\left(\Gamma^{-1}(L\Delta^{-\theta})\right)^\gamma}.
	\end{aligned}
	\end{equation}
	Hence, it follows from $u \in \mathcal T^t_{1,m} \subset [t_m, t_{m+1})$ that
	\begin{align} \nonumber
	e^{qru} |X^{K,\Delta}_t(u)-Z^{K,\Delta}_t(u)|^q 
	& \le C \mathcal H^q(\Delta) +  \frac{C_\xi e^{qru} \left( 1 + |t_{j+m}|^{q+\gamma} + |t_{j+m+1}|^{q+\gamma} \right)}{\left(\Gamma^{-1}(L\Delta^{-\theta})\right)^\gamma} \\ \nonumber
	& \le C \mathcal H^q(\Delta) +  \frac{C_{T,\xi} e^{qru} \left( 1 + |u|^{q+\gamma} + |u+1|^{q+\gamma} \right)}{\left(\Gamma^{-1}(L\Delta^{-\theta})\right)^\gamma} \\ \nonumber 
	& \le C_{T,\xi} \left( \mathcal H^q(\Delta) + \left( \Gamma^{-1}\left( L \Delta^{-\theta} \right) \right)^{-\gamma} \right)
	\end{align}
	and
	$
	E_1 \le C_{T,\xi} \left( \mathcal H^q(\Delta) + \left( \Gamma^{-1}\left( L \Delta^{-\theta} \right) \right)^{-\gamma} \right).
	$
	Similar to the estimate of $E_1$, we derive that
	$
	E_3 \le C_{T,\xi} \left( \mathcal H^q(\Delta) + \left( \Gamma^{-1}\left( L \Delta^{-\theta} \right) \right)^{-\gamma} \right).
	$
	
	\textbf{Estimates for $E_2$, $E_5$, and $E_6$.} Similar to the proof of \emph{Case $3$} in \cite[Lemma 4.6]{LLM}, we have that for any $-j \le m \le -1$ and $u \in \mathcal T^t_{1,m}$,
	\begin{equation*}
	\begin{aligned}
	& \quad  e^{ru}|  X^{K,\Delta}_t(u) -  Z^{K,\Delta}_t(u)| \le |  X^{K,\Delta}_t(u) -  Z^{K,\Delta}_t(u)|  \\
	& \le 2 | f(X^{K,\Delta}_{t_{j+m}}) \Delta |   
	+   | g(X^{K,\Delta}_{t_{j+m}}) \Delta B_{j+m} | +  | g(X^{K,\Delta}_{t_{j+m}}) \left( B(t+u) - B(t_{j+m}) \right) |,
	\end{aligned}
	\end{equation*}
	which gives
	\begin{equation}\label{E2}
	\begin{aligned}
	E_2 & \leq \EE \Big[\sup_{-j \le m \le -1} \sup_{u \in \mathcal T^t_{1,m}}  |X^{K,\Delta}_t(u)-Z^{K,\Delta}_t(u)|^q  \Big]\\
	& \le 6^q  \EE \Big[ \sup_{-j \le m \le -1} | f(X^{K,\Delta}_{t_{j+m}}) \Delta |^q \Big]
	+  3^q  \EE \Big[ \sup_{-j \le m \le -1}   | g(X^{K,\Delta}_{t_{j+m}}) \Delta B_{j+m} |^q \Big]  \\
	& \quad + 3^q   \EE \Big[ \sup_{-j  \le m \le -1} \sup_{u \in \mathcal T^t_{1,m}}  | g(X^{K,\Delta}_{t_{j+m}}) \left( B(t+u) - B(t_{j+m}) \right) |^q \Big] \\
	& =: E_{21}  +  E_{22}  +  E_{23}. 
	\end{aligned}
	\end{equation}
	By \eqref{f-lin}, $\theta \in (0,1/2)$, and Lemma \ref{th3.1}, we obtain
	\begin{equation}\label{E21}
	E_{21} \le  (12 L)^q  \Delta^{q(1-\theta)} \Big( 1 +   \EE \Big[ \sup_{0 \le m \le j-1} \|X^{K,\Delta}_{t_{m}}\|_r^q  \Big] \Big) \le C_{T,\xi} \Delta^{\frac{q}{2}}.
	\end{equation}
	Applying H\"older's inequality, Assumption \ref{a2.2}, Lemma \ref{th3.1}, and \cite[Lemma 1.2.1]{CR81}, we derive that for any $p > q$ and $\iota \in (0, q/2)$,
	\begin{equation}\label{E22}
	\begin{aligned}
	& \quad E_{22} \le 3^q \bigg( \EE \Big[ \sup_{0 \le m \le j-1}   | g(X^{K,\Delta}_{t_{m}})|^p \Big] \bigg)^{\frac{q}{p}} 
	\bigg( \EE \Big[ \sup_{0 \le m \le j-1} |\Delta B_{m}|^{\frac{pq}{p-q}} \Big] \bigg)^{\frac{p-q}{p}}  \\
	& \le C \left( 1  + \EE \Big[ \sup_{0 \le m \le j-1} \|X^{K,\Delta}_{t_m}\|_r^p \Big] \right)^{\frac{q}{p}}  \bigg(   \EE \Big[ \sup_{\substack{s_1, s_2 \in [0,T];\\ s_1-s_2 \in [0,\D]}}  |B(s_1) - B(s_2)|^{\frac{pq}{p-q}} \Big] \bigg)^{\frac{p-q}{p}}  \\
	& \le C_{T,\xi}   \Delta^{\frac{q}{2}-\iota}.
	\end{aligned}
	\end{equation}
	Similar to the estimate of $E_{22}$, using $0 \leq t+u-t_{j+m} \leq \Delta$ for $u \in \mathcal T^t_{1,m}$, we arrive at
	\begin{equation}\label{E23}
	E_{23} \le C_{T,\xi} \Delta^{\frac{q}{2}-\iota}, \quad \forall \, \iota \in (0, q/2).
	\end{equation}
	Substituting \eqref{E21}, \eqref{E22} and \eqref{E23} into \eqref{E2} gives
	$
	E_2 \le C_{T,\xi} \Delta^{\frac{q}{2}-\iota}, \, \forall \, \iota \in (0, q/2).
	$
	For $u\in\mathcal T^t_{2,m}$ with $-j \leq m \leq -1$ and for $u = 0$, similar estimates yield
	$
	E_5 \vee E_6  \le C_{T,\xi} \Delta^{\frac{q}{2}-\iota},\, \forall \, \iota \in (0, q/2).
	$
	
	\textbf{Estimate for $E_4$.} 
	Using an argument similar to that in \emph{Case $4$} of \cite[Lemma 4.6]{LLM}, together with Assumptions \ref{a2.1-f} and \ref{a2.1-g}, implies
	$$
	E_4 \le C_\xi \left( \mathcal H^{q}(\Delta)  +  \Delta^{\frac{q}{2}-\iota} \right), \quad \forall \, \iota \in (0, q/2).
	$$

	The estimates for $E_1, \ldots, E_6$, together with $Z^{K,\Delta}_t(u)=Z^{K,\Delta}_t(-K)$ and $X^{K,\Delta}_t(u)=X^{K,\Delta}_t(-K)$ for all $u \le -K$, and the monotonicity of $e^{qru}$ in $u$, yield
	$$
	\begin{aligned}
	& \quad \EE \Big[ \sup_{u \leq 0} \Big( e^{ru} |X^{K,\Delta}_t(u) - Z^{K,\Delta}_t(u) | \Big)^q \Big] \\
	& =
	\EE \Big[ \sup_{-K \le u \le 0} \Big( e^{ru} | X^{K,\Delta}_t(u) - Z^{K,\Delta}_t(u) | \Big)^q \Big] \\ & \le C_{T, \xi}  \left( \Delta^{\frac{q}{2}-\iota} + \mathcal H^q(\Delta) + \left( \Gamma^{-1}\left( L \Delta^{-\theta} \right) \right)^{-\gamma} \right),
	\end{aligned}
	$$
	where $C_{T,\xi}$ is independent of $K$, $t$ and $\Delta$.
	Furthermore,  taking the supremum over $t \in [0,T]$ and $K \in \mathbb{N}_+$ completes the proof.
\end{proof}

\begin{proof}[\textbf{Proof of Lemma \ref{lemma3.6}}]
We first prove the result for $q \ge 2$. The case of $q \in (0,2)$ follows from H\"older's inequality. Fix $T > 0$ and $K \geq T$. For any $\Delta_1 \in (0,1]$ and $\Delta \in (0,\Delta_1]$, define the stopping time
	$
	\sigma^K_{\Delta,\Delta_1} := \tau^K_{\Gamma^{-1}(L \Delta_1^{-\theta})}  \wedge  \tau^{K}_{\Delta, \Delta_1},
	$
	where $\tau^K_{\Gamma^{-1}(L \Delta_1^{-\theta})}$ and $\tau^{K}_{\Delta, \Delta_1}$ are defined in Lemma \ref{th2.3'} and \eqref{rh},  respectively. For simplicity, we write $\sigma := \sigma^K_{\D,\D_1}$. Making use of Young's inequality, we have that for any $\delta > 0$, $t \in [0, T]$ and $p > q$,
$$
	\begin{aligned}
\mathbb{E} \big[ \big\| X^{K,\Delta}_t - x^{K}_t\big\|_r^q \big] & \leq \mathbb{E} \big[ \big\| X^{K,\Delta}_t - x^{K}_t\big\|_r^q \mathbf{1}_{\{\sigma \le T\}} \big] \\
 & +  2^{q-1} \mathbb{E} \Big[ \big\| X^{K,\Delta}_t -Z^{K,\D}_t\big\|_r^q \mathbf{1}_{\{\sigma > T\}}+ \big\| Z^{K,\Delta}_t - x^{K}_t\big\|_r^q \mathbf{1}_{\{\sigma > T\}} \Big]. 
	\end{aligned}
$$
Lemmas \ref{th3.1}, \ref{th2.3'}, and \ref{l3.5} imply
\begin{equation}\label{eq3.18}
\begin{aligned}
 & \quad\,\, \mathbb{E} \big[ \big\| X^{K,\Delta}_t - x^{K}_t\big\|_r^q \mathbf{1}_{\{\sigma \le T\}} \big] \\
 & \le \frac{q \delta}{p}  \mathbb{E} \big[ \big\| X^{K,\Delta}_t - x^{K}_t\big\|_r^p \big] + \frac{p-q}{p \delta^{q/(p-q)}}  \mathbb{P}{\{\sigma \le T\}} \\
 & \le \frac{2^{p-1} q \delta}{p} \left( \EE \big[ \big\| X^{K,\Delta}_t\big\|_r^p \big] + \EE \big[ \big\| x^{K}_t\big\|_r^p \big] \right) \\
 & \quad + \frac{p-q}{p \delta^{q/(p-q)}} \Big( \PP \{  \tau^K_{\Gamma^{-1}(L \Delta_1^{-\theta})}\le T  \} + \PP \{ \tau^{K}_{\Delta,\Delta_1}  \le T \} \Big) \\
 & \le C_{T,\xi}\delta + \frac{C_{T,\xi}}{\delta^{q/(p-q)} (\Gamma^{-1} \left(L \Delta_1^{-\theta})\right)^p}.
\end{aligned}
\end{equation}
According to Lemma \ref{l3.6}, we obtain that for any $\varepsilon_1 \in (0, q/2)$ and $\gamma > 0$,
	$$
	\begin{aligned}
	  \mathbb{E} \big[ \big\| X^{K,\Delta}_t -Z^{K,\D}_t\big\|_r^q \mathbf{1}_{\{\sigma > T\}} \big]
	& = \EE \bigg[ \bigg( \sup_{u \le 0} \Big( e^{ru}  | X^{K,\Delta}_t(u) -Z^{K,\D}_t(u)| \Big)^q \bigg)  \mathbf{1}_{\{\sigma > T\}} \bigg]  \\
	& \le C_{T,\xi}  \left( \Delta^{\frac{q}{2}-\varepsilon_1} + \mathcal H^q(\Delta) + \left( \Gamma^{-1}\left( L \Delta^{-\theta} \right) \right)^{-\gamma} \right).
	\end{aligned}
	$$
	Hence,
	\begin{equation}\label{3.29}
	\begin{aligned}
	\mathbb{E} \big[ \big\| X^{K,\Delta}_t - x^{K}_t\big\|_r^q \big] 
	& \le C_{T,\xi} \delta  +  \frac{C_{T,\xi}}{\delta^{q/(p-q)} (\Gamma^{-1}(L \Delta_1^{-\theta}))^p} \\
	& \quad + C_{T,\xi} \left( \Delta^{\frac{q}{2}-\varepsilon_1} + \mathcal H^q(\Delta) + \left( \Gamma^{-1}\left( L \Delta^{-\theta} \right) \right)^{-\gamma}\right) \\
	& \quad + 2^{q-1} \mathbb{E} \big[  \big\| Z^{K,\Delta}_t - x^{K}_t\big\|_r^q \mathbf{1}_{\{\sigma > T\}} \big].
	\end{aligned}
	\end{equation}
	Furthermore, by the definition of $\|\cdot\|_r$ and $K \ge T$, one has
	\begin{equation*}
	\begin{aligned}
	\big\|Z^{K,\Delta}_t-x^K_t\big\|^q_r \II_{\{\sigma > T\}} \le
	&   \sup_{u \le -K}  \left( e^{ru}   \left| \xi(t-K)   - \xi(t + u) \right| \right)^q  \\
	& + \sup_{-K \le u \le 0}  \left( e^{ru}  \big|  Z^{K,\Delta}_t(u) - x^K_t(u)  \big|  \right)^q \II_{\{\sigma > T\}}. \\
	\end{aligned}
	\end{equation*}
It follows from the uniform continuity of $\xi$ that 
\begin{equation}\label{eq:3.23}
\begin{aligned}
 \sup_{u\le-K}\left(e^{ru}|\xi(t-K)-\xi(t+u)|\right)^q  \le e^{-qrK}\sup_{s\ge0}e^{-qrs}\mathcal H^q(s)  \le C_\xi e^{-qrK}.
\end{aligned}
\end{equation}
	Therefore,
	$
	\big\|Z^{K,\Delta}_t-x^K_t\big\|^q_r \II_{\{\sigma > T\}}   \le   C_\xi e^{-qrK} +  \sup_{0 \le u \le t}  \big| Z^{K,\Delta}(u) - x^{K}(u)  \big|^q \II_{\{\sigma > T\}}.
	$
By Assumptions \ref{a2.1-f} and \ref{a2.1-g}, Lemma \ref{l3.6}, and the argument in \cite[Lemma 3.4]{LLM}, we obtain
	\begin{equation}\label{3.30}
	\begin{aligned}
	& \quad\,\, \mathbb{E} \big[ \big\| Z^{K,\Delta}_t - x^{K}_t\big\|_r^q \mathbf{1}_{\{\sigma > T\}} \big] \\
	& \le C_{T ,\xi} e^{-qrK} +  \EE \Big[ \sup_{0 \le u \le t}  \big| Z^{K,\Delta}(u) - x^{K}(u)  \big|^q \mathbf{1}_{\{\sigma > T\}} \Big] \\
	& \le C_{T ,\xi} e^{-qrK} +  C_{T,\xi,\D_1} \left( \mathcal H^q(\D)  + \D^{1\we\frac{q}{4}} + \left( \Gamma^{-1} \left( L \Delta^{-\theta} \right) \right)^{-\gamma} \right).
	\end{aligned}
	\end{equation}
	Inserting \eqref{3.30} into \eqref{3.29} gives
	\begin{equation}\label{eq3.23}
	\begin{aligned}
	& \quad\,\, \mathbb{E} \big[ \big\| X^{K,\Delta}_t - x^{K}_t\big\|_r^q \big]
 \\
	& \le C_{T,\xi} \delta  +  \frac{C_{T,\xi}}{\delta^{q/(p-q)} \left(\Gamma^{-1}(L \Delta_1^{-\theta})\right)^p} + C_{T,\xi} e^{-qrK}  \\
	& \quad + C_{T,\xi, \Delta_1} \left( \Delta^{\frac{q}{2}-\varepsilon_1} + \Delta^{1 \wedge \frac{q}{4}} + \mathcal H^q(\Delta) + \left( \Gamma^{-1} \left( L \Delta^{-\theta} \right) \right)^{-\gamma}\right) .  \\
	\end{aligned}
	\end{equation}
First letting $K \to +\infty$ and $\Delta \to 0$, then letting $\Delta_1 \to 0$, and finally letting $\delta \to 0$ proves the assertion.
\end{proof}
\end{appendices}

\bibliography{references}

\end{document}